%% file: main.tex
\documentclass[11pt,reqno]{amsart}

\usepackage[T1]{fontenc}
\usepackage[boldsans]{ccfonts}
\usepackage{eulervm}
\renewcommand{\mathbf}{\mathbold}
\usepackage{bbm}

\usepackage{enumerate, amsmath, amsthm, amsfonts, amssymb, 
mathrsfs, graphicx, paralist} 
\usepackage[usenames, dvipsnames]{color}
\usepackage[margin=1in]{geometry} 
\usepackage[hidelinks]{hyperref}
\usepackage{comment}

\usepackage{mathabx} 
\usepackage{tikz-cd}

\numberwithin{equation}{section}
\newtheorem{Theorem}[equation]{Theorem}
\newtheorem{Proposition}[equation]{Proposition}
\newtheorem{Lemma}[equation]{Lemma}
\newtheorem{Corollary}[equation]{Corollary}

\newtheorem{Problem}[equation]{Problem}

\theoremstyle{definition}
\newtheorem{Remark}[equation]{Remark}

\newtheorem{eg}[equation]{Example}

\newtheorem{Definition}[equation]{Definition}

\newcommand{\cA}{\mathcal{A}}

\newcommand{\cF}{\mathcal{F}}

\newcommand{\bG}{\mathbf{G}}

\newcommand{\cK}{\mathcal{K}}

\newcommand{\bL}{\mathbf{L}}

\newcommand{\cM}{\mathcal{M}}

\newcommand{\bN}{\mathbf{N}}
\newcommand{\cN}{\mathcal{N}}

\newcommand{\cO}{\mathcal{O}}

\newcommand{\cR}{\mathcal{R}}

\newcommand{\bT}{\mathbf{T}}

\newcommand{\cW}{\mathcal{W}}

\newcommand{\cX}{\mathcal{X}}

\newcommand{\cZ}{\mathcal{Z}}

\newcommand{\bg}{\mathbf{g}}

\newcommand{\bm}{\mathbf{m}}

\newcommand{\bu}{\mathbf{u}}

\newcommand{\bv}{\mathbf{v}}

\newcommand{\bw}{\mathbf{w}}

\newcommand{\bz}{\mathbf{z}}

\renewcommand{\AA}{\mathbb{A}}

\newcommand{\CC}{\mathbb{C}}

\newcommand{\GG}{\mathbb{G}}

\newcommand{\PP}{\mathbb{P}}

\newcommand{\WW}{\mathbb{W}}
\newcommand{\XX}{\mathbb{X}}

\newcommand{\ZZ}{\mathbb{Z}}

\renewcommand{\phi}{\varphi}
\renewcommand{\emptyset}{\varnothing}

\renewcommand{\tilde}[1]{\widetilde{#1}}

\makeatletter
\def\Ddots{\mathinner{\mkern1mu\raise\p@
\vbox{\kern7\p@\hbox{.}}\mkern2mu
\raise4\p@\hbox{.}\mkern2mu\raise7\p@\hbox{.}\mkern1mu}}
\makeatother

\newcommand{\hooklongrightarrow}{\lhook\joinrel\longrightarrow}

\DeclareMathOperator{\Hom}{Hom}

\DeclareMathOperator{\Spec}{Spec}

\newcommand{\GL}{\mathbf{GL}}
\newcommand{\SL}{\mathbf{SL}}

\newcommand{\Gr}{\mathrm{Gr}}

\newcommand{\fsl}{\mathfrak{sl}}

\newcommand{\suchthat}{:}
\newcommand{\val}{\mathrm{val}}

\newcommand{\pt}{\mathrm{pt}}

\newcommand{\kk}{\Bbbk}

\renewcommand{\bv}{\mathbf{v}} 
\renewcommand{\bw}{\mathbf{w}} 
\newcommand{\su}{\mathsf{u}} 
\newcommand{\sv}{\mathsf{v}} 
\newcommand{\sw}{\mathsf{w}} 
\newcommand{\sm}{\mathsf{m}} 
\newcommand{\loc}{\mathrm{loc}} 
\newcommand{\M}{\mathsf{M}}

\newcommand{\source}{\mathsf{s}}
\newcommand{\target}{\mathsf{t}}
\newcommand{\arrowset}{E}

\newcommand{\bNone}{\bN_1}
\newcommand{\bNtwo}{\bN_2}

\newcommand{\simplecoroot}{\alpha^\vee}
\newcommand{\simpleroot}{\alpha}
\newcommand{\fundweight}{\Lambda}

\newcommand{\weightlattice}{P}
\newcommand{\corootlattice}{Q^\vee}
\newcommand{\coweightlattice}{P^\vee}

\newcommand{\zastavap}{\cZ_+}
\newcommand{\zastavam}{\cZ_-}
\newcommand{\bbzastavap}{\ZZ_+}
\newcommand{\bbzastavam}{\ZZ_-}

\begin{document}

\title[FMOs and embeddings of Kac-Moody affine Grassmannian slices]{Fundamental monopole operators and embeddings of Kac-Moody affine Grassmannian slices}
\author{Dinakar Muthiah}
\address{D.~Muthiah:
School of Mathematics and Statistics,
University of Glasgow,
University Place,
Glasgow G12 8QQ,
Scotland
}
\email{dinakar.muthiah@glasgow.ac.uk }
\author{Alex Weekes}
\address{A.~Weekes: Department of Mathematics and Statistics, University of Saskatchewan, Canada}
\email{weekes@math.usask.ca}

\maketitle

\begin{abstract}

  Braverman, Finkelberg, and Nakajima define \emph{Kac-Moody affine Grassmannian slices} as Coulomb branches of $3d$ $\cN=4$ quiver gauge theories and prove that their Coulomb branch construction agrees with the usual loop group definition in finite ADE types. The Coulomb branch construction has good algebraic properties, but its geometry is hard to understand in general.

 In finite types, an essential geometric feature is that slices embed into one another. We show that these embeddings are compatible with the \emph{fundamental monopole operators} (FMOs), remarkable regular functions arising from the Coulomb branch construction.  Beyond finite type these embeddings were not known, and our second result is to construct them for all symmetric Kac-Moody types. We show that these embeddings respect Poisson structures under a mild ``goodness'' hypothesis.  These results give an affirmative answer to a question posed by Finkelberg in his 2018 ICM address and demonstrate the utility of FMOs in studying the geometry of Kac-Moody affine Grassmannian slices, even in finite types.

\end{abstract}

\input{intro}
\input{finite}
\input{beyondfinite}
\input{appendix}

\bibliographystyle{amsalpha}
\bibliography{references}

\end{document}

%% file: intro.tex
\section{Introduction}

\subsection{Kac-Moody Affine Grassmannian slices} 
Braverman, Finkelberg, and Nakajima \cite{BFN2} proved the spectacular result that affine Grassmannian slices of finite ADE type arise as Coulomb branches of $3d$ $\cN=4$ quiver gauge theories. These slices $\overline{\cW}_\mu^\lambda$ model the singularities between spherical Schubert varieties in the affine Grassmannian, and thereby encode essential aspects of the geometric Satake correspondence.  Here $\lambda$ and $\mu$ are coweights with $\lambda$ dominant and $\lambda \geq \mu$.

Remarkably, the Braverman-Finkelberg-Nakajima construction works for arbitrary quivers. One can thus take the corresponding Coulomb branches as a definition of \emph{Kac-Moody affine Grassmannian slices} $\overline{\cW}_\mu^\lambda$ for general symmetric Kac-Moody types.  Starting with the work of Braverman and Finkelberg \cite{BF}, these slices are proposed as a fruitful setting for developing the geometric Satake correspondence for Kac-Moody groups \cite{Nakajima-Branching,
  BF-II, BF-III,
  BFN2,FinkelbergICM, NakSatake}.

The Coulomb branch is defined as the spectrum of a certain convolution algebra (see \S \ref{subsection: Coulomb branches in general} and the references within). Despite the abstract definition, the construction is well suited for deducing algebraic properties of Kac-Moody affine Grassmannian slices (via general results for Coulomb branches from \cite{BFN1}).  However, the geometry of these varieties remains quite mysterious.  For example, in Finkelberg's 2018 ICM address \cite[\S 5.6 and \S 6.4]{FinkelbergICM}, he sought certain embeddings $\overline{\cW}_\mu^{\lambda'} \hookrightarrow \overline{\cW}_\mu^\lambda$ for general Kac-Moody types. These embeddings are well-known in finite types, but even in this case the relationship with Coulomb branches has been mysterious (see \S \ref{intro: finite type} below for more detail).

The goal of this paper is to remedy this situation. In finite types, we show that the above closed embeddings respect certain functions called \emph{fundamental monopole operators} which arise naturally in the Coulomb branch picture (see Theorem \ref{thm:finite-type-embeddings-compatible-with-fmos-intro}).  For general symmetric Kac-Moody type we exhibit closed embeddings $\overline{\cW}_\mu^{\lambda'} \hookrightarrow \overline{\cW}_\mu^\lambda$ (see Theorem \ref{thm:infinite-type-embedding-intro}), and study their interaction with Poisson structures (see Theorem \ref{thm:poisson-subvariety-intro}).

\subsection{Embeddings of slices in finite type}
\label{intro: finite type}
Let us restrict our attention to finite ADE types.  One geometric phenomenon that is obvious from the loop group-theoretic picture of affine Grassmannian slices is the fact that slices embed into one another: for $\lambda'$ another dominant coweight with $\lambda \geq \lambda' \geq \mu$, there is a closed embedding:
\begin{equation}
  \label{eq:d:109}
 \overline{\cW}^{\lambda'}_\mu \hookrightarrow  \overline{\cW}^\lambda_\mu
\end{equation}
For readers familiar with affine Grassmannians, this is analogous to and indeed a direct consequence of the fact that the spherical Schubert variety $\overline{\Gr^{\lambda'}}$ embeds into $\overline{\Gr^\lambda}$. 

This closed embedding, which is a basic feature of the geometry, is mysterious from the perspective of Coulomb branches. Moreover, any Coulomb branch construction of this embedding must necessarily be subtle: all of the basic geometric operations on the convolution algebra defining a Coulomb branch lift to its quantization, but algebraically  one can easily see that the closed embedding \eqref{eq:d:109} need not lift to the quantization (see Remark \ref{rem:poisson-subvariety-does-not-quantize}).

\subsection{Fundamental monopole operators}
Our first goal is to understand how the closed embedding \eqref{eq:d:109} interacts with the Coulomb branch descriptions of both sides. To this end, we study its effect on \emph{fundamental monopole operators} (FMOs).  The FMOs are regular functions on $\overline{\cW}_\mu^\lambda$ that are natural from the Coulomb branch perspective, corresponding to certain algebraic cycles  in the convolution algebra defining the Coulomb branch.  A priori, these functions have nothing to do with the closed embedding \eqref{eq:d:109}. Nonetheless, we prove the following result:

\begin{Theorem}
  \label{thm:finite-type-embeddings-compatible-with-fmos-intro}
Under the closed embedding \eqref{eq:d:109}, FMOs map to FMOs or to zero.  
\end{Theorem}
See Theorem \ref{thm:compatibility-of-FMOs-with-embeddings-of-slices} in the main text for the precise formula. By work of the second author \cite{Weekes1}, the FMOs generate the coordinate ring of $\overline{\cW}^\lambda_\mu$, so the closed embedding \eqref{eq:d:109} is completely characterized by this theorem.

There are explicit rational expressions (see \eqref{def: FMOs}) for the FMOs, which arise from the localization theorem.  Unfortunately, the denominators in these rational expressions vanish on the image of the closed embedding \eqref{eq:d:109}, so it is not clear how they restrict.  To prove the theorem, we move the problem to studying the ``adding defect'' map of \emph{zastava spaces} (see \S \ref{sec: adding defect}) where the denominators can be controlled.

\subsection{Embeddings of slices in Kac-Moody type}
Theorem \ref{thm:finite-type-embeddings-compatible-with-fmos-intro} suggests a natural guess for how to define a closed embedding of Kac-Moody affine Grassmannian slices: we should map FMOs using the same formula. Unfortunately, we do not know the relations among the FMOs, so it is unclear that such a map extends to an algebra homomorphism. Indeed, this is an important open problem.

\begin{Problem}
  \label{problem:fmo-relations}
  Describe a complete set of relations among the fundamental monopole operators.
\end{Problem}

If we had a solution to this problem, we could proceed by simply stating where FMOs map and checking relations. Instead, we  construct the map another way, which is our second main result.

\begin{Theorem}
  \label{thm:infinite-type-embedding-intro}
  Suppose further that $\lambda'$ satisfies a certain ``conicity condition'' with respect to $\lambda$. Then there is a closed embedding of Kac-Moody affine Grassmannian slices
  \begin{equation}
    \label{eq:d:110}
 \overline{\cW}^{\lambda'}_\mu \hookrightarrow  \overline{\cW}^\lambda_\mu
  \end{equation}
  under which the FMOs map to FMOs or to zero.
\end{Theorem}

See Theorem \ref{thm:closed-embedding-of-KM-slices} and \ref{thm:big-tikz-diagram-defining-closed-embedding} for precise statements. We mention that the conicity condition (see Definition \ref{def:conical-coweights}) is automatic in finite-type, and it holds in affine type for all non-zero levels, which are the cases of primary interest (see Theorem \ref{thm: conicity finite affine}). 
We also note that Theorem \ref{thm:infinite-type-embedding-intro} does not subsume Theorem \ref{thm:finite-type-embeddings-compatible-with-fmos-intro}. We only know that the two maps agree in finite type because of our result showing that FMOs map the same way in both cases.

To construct the map in the theorem we make use of several natural Coulomb branch operations from \cite{BFN1}. However, we require an additional map. The coordinate rings of Coulomb branches carry gradings corresponding to the \emph{monopole formula}, and the conicity condition guarantees that $\overline{\cW}^{\lambda}_{\lambda'}$  is a cone with respect to this grading (see Lemma \ref{lemma: first lemma appendix}). We require the embedding of the cone point into this cone. We do not know of any natural Coulomb branch construction of the cone point (e.g. it does not lift to the quantization in general) but we are able to proceed by simply appealing to the grading. Finally, we verify that the above map has the correct behavior on FMOs. Also, our map is a priori only rationally defined, and only because of our FMO calculation do we conclude that the map is regular.

Theorems \ref{thm:finite-type-embeddings-compatible-with-fmos-intro} and \ref{thm:infinite-type-embedding-intro} show that the geometry of Kac-Moody affine Grassmannian slices is closely coupled with the FMOs, and we propose that they should be standard tools for future investigation. In particular, understanding Problem \ref{problem:fmo-relations} is crucial.

\subsection{Symplectic leaves and Poisson subvarieties}

In \cite{Weekes2}, the second author proved that Coulomb branches of quiver gauge theories have symplectic singularities for all quivers with no loops and no multiple edges. By work of Kaledin \cite{Kaledin}, we conclude that these Coulomb branches have finitely many closed Poisson subvarieties, and their symplectic leaves are obtained precisely as the smooth loci of irreducible closed Poisson subvarieties. We prove the following.

\begin{Theorem}[Theorem \ref{thm:closed-poisson-embedding-of-KM-slices} in the main text]
  \label{thm:poisson-subvariety-intro}
  Assume that $\lambda'$ satisfies a certain ``good condition'' for $\lambda$. Then the closed embedding
  \begin{equation}
    \label{eq:d:64}
 \overline{\cW}^{\lambda'}_\mu \hookrightarrow  \overline{\cW}^\lambda_\mu
  \end{equation}
  is compatible with Poisson structures. In particular, the inclusion of the smooth locus $(\overline{\cW}_\mu^{\lambda'})^{reg} \hookrightarrow \overline{\cW}_\mu^\lambda$ defines a symplectic leaf.
\end{Theorem}

We mention that the good condition (see Definition \ref{def:good-coweights}) is automatic in finite-type, and it holds in affine type for all levels greater than or equal to two (see Theorem \ref{thm: conicity finite affine}).

In our previous work \cite{MuthiahWeekes}, we constructed all symplectic leaves (and therefore all closed Poisson subvarieties) for finite-type slices, and indeed they are exactly the closed embeddings we study here. However, in infinite Kac-Moody type it is known that there are further closed Poisson subvarieties.

For an arbitrary Coulomb branch, Nakajima has made precise predictions for the enumeration of symplectic leaves based on symplectic duality \cite[\S 2]{NakQuestions}.  For quiver gauge theories -- and thus for Kac-Moody affine Grassmannian slices -- this enumeration should be dual to the enumeration of symplectic leaves for Nakajima quiver varieties. A complete enumeration of the latter is given by Bellamy and Schedler \cite[Theorem 1.9]{BS}. It would be very interesting to generalize Theorems \ref{thm:infinite-type-embedding-intro} and  \ref{thm:poisson-subvariety-intro} to obtain all of the ``dual'' symplectic leaves on the affine Grassmannian side.

\subsection{Previous work in affine type A}
\label{sec: affine type A}
In addition to finite type, there is another important case where the geometry of affine Grassmannian slices is well understood: affine type A. In this case Nakajima and Takayama \cite{NT} prove that the Kac-Moody affine Grassmannian slices are isomorphic to Cherkis bow varieties. In particular, they obtain an explicit geometric invariant theory description of these varieties and from that a full description of their symplectic leaves, transversal slices, and torus action. In later work \cite{NakSatake}, Nakajima has proved further geometric facts about bow varieties and has succeeded in generalizing the geometric Satake correspondence in this setting by geometrically constructing the irreducible representations of affine type A.

From Nakajima and Takayama's work, we see that our construction recovers many, but not all, of the symplectic leaves \cite[Theorem 7.26]{NT}. Furthermore, in describing symplectic leaves, they discover a distinction between level one and levels greater than one. In particular, the good hypothesis in Theorem \ref{thm:poisson-subvariety-intro} is sharp in this setting (cf.~Theorem \ref{thm: conicity finite affine}).

\subsection{Relation to physics}
Coulomb branches of $3d$ $\cN=4$ gauge theories originate in theoretical physics, see for example \cite{NakTowards} and references therein.  In addition to these physical origins, an important physically-motivated ingredient in the present work is the \emph{monopole formula} \cite{CHZ}, which encodes the Hilbert series of the coordinate ring of a  Coulomb branch with respect to a certain natural grading.  See Appendix \ref{sec:appendix} for an overview. In particular, the ``good condition'' from Theorem \ref{thm:poisson-subvariety-intro} corresponds to being a good theory in physics terminology, while the ``conicity condition'' from Theorem \ref{thm:infinite-type-embedding-intro} corresponds to a good or ugly theory.  

The study of symplectic leaves of Coulomb branches is also of physical significance, see for example the works of Hanany and collaborators \cite{BCGHSZ,GH,CH}.

\subsection{Outline of the paper} 
In \S 2, we prove Theorem \ref{thm:compatibility-of-FMOs-with-embeddings-of-slices}: closed embeddings of finite-type slices are compatible with FMOs. Aside from the formula for the FMOs, Coulomb branches do not enter the discussion in this section.  In \S 3, we recall the fundamentals of Coulomb branches and the definition of Kac-Moody affine Grassmannian slices as Coulomb branches of quiver gauge theories. In \S 4, we construct the embeddings of Kac-Moody affine Grassmannian slices and prove Thoerem \ref{thm:closed-embedding-of-KM-slices}. We also explain in \S 4.3 about compatibility with Poisson structure. Finally, in Appendix \ref{sec:appendix} we overview the monopole formula and its consequences for quiver gauge theories.

\subsection{Notation}

For a positive integer $\sv$, we write $[\sv] = \left\{1,\cdots,\sv\right\}$. We will work with schemes and ind-schemes over a field $\kk$. By a variety we mean a integral scheme of finite-type over $\kk$. Let $I$ be a finite set (later it will be the vertex set of a quiver), and let $\ZZ^I$ denote the set of $I$-tuples of integers. We write $\mathbf{0}  \in \ZZ^I$ for the $I$-tuple consisting of all zeros. Given two elements $\bm  = (\sm_i)_{i \in I}$ and $\bv = (\sv_i)_{i \in I}$ of $\ZZ^I$, we write $\bm \leq \bv$ to mean $\sm_i \leq \sv_i$ for all $i \in I$.  Given $\bv \in \ZZ^I$, we denote by $S_\bv = \prod_{i \in I} S_{\sv_i}$ the corresponding product of symmetric groups.
\subsection{Acknowledgements}

D.M.~was supported by JSPS KAKENHI Grant Number JP19K14495. A.W.~was supported by an NSERC Discovery Grant. We are grateful to the organizers of the conference ``Bundles and Conformal Blocks with a Twist'' which took place in June 2022 at the ICMS where some of this work was completed.


%% file: finite.tex
\section{Affine Grassmannian slices}
\label{sec:generalized-affine-grassmannian-slices}

Let $G$ be a split semisimple \emph{simply-laced} and simply-connected group with a fixed pinning.  We make the simply-connectedness assumption so that $G$ has all fundamental representations, but this is largely not necessary (see Remark \ref{rem:simply-connected-assumption} below). The pinning determines a pair $B^+$ and $B^-$ of opposite Borel subgroups. Let $U^+$ and $U^-$ be the unipotent radicals of $B^+$ and $B^-$, and let $T = B^+ \cap B^-$, a maximal torus. 

Let $I$ denote the vertices of the Dynkin diagram. Let $\{\simpleroot_i\}_{i\in I}$ denote the set of simple roots, and $\{\fundweight_i\}_{i \in I}$ denote the set of fundamental weights.  Let $\weightlattice$ denote the weight lattice of $T$, and $\weightlattice_{++}$ the set of dominant weights.  Let $\{ \simplecoroot_i\}_{i \in I}$ denote the set of simple coroots, and let $\corootlattice_+=\bigoplus_{i \in I} \ZZ_{\geq 0} \simplecoroot_i$ denote the positive coroot cone. Let $\coweightlattice$ denote the  coweight lattice of the torus $T$, and let $\coweightlattice_{++}$ denote the set of dominant coweights.   Finally, recall that the dominance order on $\coweightlattice$ is defined by $\lambda \geq \mu$ iff $\lambda - \mu \in \corootlattice_+$.

For each dominant weight $\Lambda \in \weightlattice_{++}$, let $W(-\Lambda)$ denote the Weyl module with lowest weight $-\Lambda$, and let $S(-\Lambda)$ denote the Schur module with lowest weight $-\Lambda$. We interpret both of these as \emph{left} representations of $G$, and we interpret their vector-space duals as \emph{right} representations. Write $V(-\Lambda)$ for a representation that is either $W(-\Lambda)$ or $S(-\Lambda)$, but where we have not specified which. The $-\Lambda$ weight space $V(-\Lambda)_{-\Lambda}$ is one dimensional, and we fix a generator $|v_{-\Lambda}\rangle$. Let $\langle v^*_{-\Lambda} |$ be the unique weight covector that pairs to $1$ with $|v_{-\Lambda}\rangle$. We use the ``bra'' and ``ket'' notation to emphasize on which side $G$ acts. 
The pinning on $G$ determines, for each $i \in I$, a fixed choice of $|v_{-s_i(\Lambda)} \rangle \in V(-\Lambda)_{-s_i(\Lambda)}$ and unique weight covector $\langle v^*_{-s_i(\Lambda)}|$ that pairs to $1$ with $|v_{-s_i(\Lambda)} \rangle$. 

\subsubsection{Colored divisors}
\label{sec:colored-divisors}
Let $\gamma \in \corootlattice_+$, which we may write as $\gamma = \sum_{i \in I} \sv_i \simplecoroot_i$ where $\sv_i = \langle \gamma, \fundweight_i \rangle$. Write $\AA^{(\gamma)}$ for the variety of $I$-colored divisors of degree $\gamma$ on $\AA^1$. Explicitly, points of $\AA^{(\gamma)}$ are given by $I$-tuples of monic polynomials $(L_i(z))_{i \in I}$ such that $L_i(z)$ has degree $\sv_i$ for each $i \in I$. The non-leading coefficients of $L_i(z)$ are regular functions on $\AA^{(\gamma)}$, and the coordinate ring $\kk[\AA^{(\gamma)}]$ is a polynomial ring in these non-leading coefficients.

\subsubsection{Quiver orientation}

The definition of fundamental monopole operators will require us to orient the edges of the Dynkin diagram.  We choose a quiver $(I,\arrowset)$ whose underlying unoriented graph is our Dynkin diagram. Here $I$ is the vertex set of the quiver, and $\arrowset$ is the arrow set. On the arrow set $\arrowset$, we have source and target maps $\source,\target: \arrowset \rightarrow I$.

\subsection{Affine Grassmannian slices}
We will follow the notation in \cite{MuthiahWeekes} closely. Let $G((z^{-1}))$ denote the Laurent series loop group of $G$ (in $z^{-1}$), and let $G[z]$ denote the subgroup of positive loops. We have an embedding $\coweightlattice \hookrightarrow G((z^{-1}))$ of the coweight lattice, denoted by $\mu \mapsto z^\mu$. 
For any affine algebraic group $H$ we may also consider the subgroup $H[[z^{-1}]] \subset H((z^{-1}))$ of negative loops. Denote by $H_1[[z^{-1}]]$ the kernel of the ``evaluation at $z=\infty$'' map to $H$, i.e. we have the short exact sequence of groups:
\begin{equation}
1 \longrightarrow H_1[[z^{-1}]] \longrightarrow H[[z^{-1}]] \longrightarrow H \longrightarrow 1
\end{equation}

For the rest of this section, fix coweights $\lambda \in \coweightlattice_{++}$ and $\mu \in \coweightlattice$ such that $\lambda \geq  \mu$. Define elements $\bv = (\sv_i)_{i \in I}$ and $\bw = (\sw_i)_{i \in I}$ of $\ZZ^I$ by
\begin{align}
  \label{eq:d:104}
  \sv_i = \langle \lambda - \mu, \fundweight_i \rangle \\
  \label{eq:d:104b}
  \sw_i = \langle \lambda,  \simpleroot_i \rangle
\end{align}
for all $i \in I$. Note that all $\sw_i, \sv_i \geq 0$, and that $\lambda - \mu = \sum_{i\in I} \sv_i \simplecoroot_i$.

Define a closed subscheme of $G((z^{-1}))$ by
\begin{equation}
  \label{eq:d:2}
  \cW_{\mu} = U^+_1[[z^{-1}]]T_1[[z^{-1}]] z^\mu U^-_1[[z^{-1}]],
\end{equation}
Consider also the closed sub-ind-scheme $\overline{\cX^{\lambda}} = \overline{G[z] z^\lambda G[z]} \subset G((z^{-1}))$, defined as the preimage of the Schubert variety $\overline{G[z] z^\lambda G[z] / G[z]} \subset G((z^{-1})) / G[z]$ in the thick affine Grassmannian under the natural map $G((z^{-1})) \rightarrow G((z^{-1})) / G[z]$.

The \emph{affine Grassmannian slice} is defined by
\begin{align}
\overline{\cW}^\lambda_\mu = \overline{\cX^\lambda} \cap \cW_{\mu}  
\end{align}
We will often write $x a z^\mu y \in \overline{\cW}^\lambda_\mu$ for elements where $x \in U^+_1[[z]]$, $a \in T_1[[z]]$, and $y \in U^-_1[[z^{-1}]]$.

Finally, we write
\begin{equation}
  \label{eq:d:6}
  \cA(\lambda,\mu) = \kk[ \overline{\cW}^\lambda_\mu]
\end{equation}
for the coordinate ring of $\overline{\cW}^\lambda_\mu$.

\begin{Remark}
When $\mu$ is dominant, the variety $\overline{\cW}_\mu^\lambda $ embeds into the affine Grassmannian of $G$, providing a transverse slice between spherical Schubert varieties \cite{KWWY}. This motivates the name ``affine Grassmannian slice''.  Note that when $\mu$  is not dominant $\overline{\cW}_\mu^\lambda$ is called a \emph{generalized} affine Grassmannian slice in \cite{BFN2}.  In this paper, we have elected to drop the word ``generalized''.
\end{Remark}

\subsubsection{A variant using matrix coefficients}

Fix $\lambda$ and $\mu$ as above. We define a closed sub-ind-scheme $\overline{\XX^\lambda} \subset G((z^{-1}))$ by
\begin{equation}
  \label{eq:d:3}
   \overline{\XX^\lambda} = \Big\{ g \in G((z^{-1})) \suchthat \val \langle u| g |v \rangle \geq - \langle \lambda, \Lambda \rangle \ \text{ for all } \ |v\rangle  \in V(-\Lambda), \  \langle u| \in V(-\Lambda)^* \Big\}
\end{equation}
where $\Lambda \in \weightlattice_{++}$ varies over all dominant weights and $V(-\Lambda)$ varies over $W(-\Lambda)$ and $S(-\Lambda)$.

It is easy to see that we have a closed embedding $\overline{\cX^\lambda} \hookrightarrow \overline{\XX^\lambda}$ that is a bijection on points, i.e. $\overline{\XX^\lambda}$ is a possibly non-reduced thickening of $\overline{\cX^\lambda}$. Define:
\begin{equation}
  \label{eq:d:4}
  \overline{\WW}^\lambda_\mu = \overline{\XX^\lambda} \cap \cW_{\mu}
\end{equation}
Then there is a closed embedding $\overline{\cW}^\lambda_\mu \hookrightarrow \overline{\WW}^\lambda_\mu$ that is a bijection on points. Because $\overline{\cW}^\lambda_\mu$ is known to be reduced \cite[\S 2(ii)]{BFN2}, we could alternatively define $\overline{\cW}^\lambda_\mu$ as the reduced scheme of $\overline{\WW}^\lambda_\mu$.

\subsection{Fundamental monopole operators}

\subsubsection{Birational coordinates}
Recall from \eqref{eq:d:104} that $\sv_i = \langle \lambda - \mu, \fundweight_i \rangle$. For any $g \in G((z^{-1}))$, we can consider $z^{\langle \lambda, \fundweight_i \rangle } \langle v_{-\fundweight_i}^* | g | v_{-\fundweight_i} \rangle$, which is a Laurent series in $z^{-1}$. If we further restrict that $g = x a z^\mu y  \in \overline{\cW}^\lambda_\mu$, then $z^{\langle \lambda, \fundweight_i \rangle }\langle v_{-\fundweight_i}^* | g | v_{-\fundweight_i} \rangle$ is a monic polynomial in $z$ of degree $\sv_i$, which we denote $Q_i(z)$ and think of as a regular function on $\overline{\cW}^\lambda_\mu$ (with values in monic polynomials of degree $\sv_i$).  Observe that $Q_i(z)$ depends only on the $a$-factor of $g = x a z^{\mu} y$. We obtain a map $\overline{\cW}^\lambda_\mu \rightarrow \AA^{(\lambda - \mu)}$ given sending a point of $\overline{\cW}^\lambda_\mu$ to the colored divisor $(Q_i(z))_{i \in I}$.

Similarly we define $P_i(z) = z^{\langle \lambda, \fundweight_i \rangle }\langle v_{-s_i(\fundweight_i)}^* | x a z^\mu y | v_{-\fundweight_i} \rangle$.  This is a (not necessarily monic) polynomial of degree less than or equal to $\sv_i - 1$, and we think of $P_i(z)$ as a regular function on $\overline{\cW^\lambda_\mu}$. Observe that $P_i(z)$ depends only on the $x$ and $a$-factors of $g = x a z^{\mu} y$. Finally, it is known that the coefficients of the $Q_i(z)$ and $P_i(z)$ for $i \in I$ form a system of birational coordinates on $\overline{\cW}_\mu^\lambda$ \cite{FKMM,BDF}, cf.~\S \ref{sec: zastava spaces} below.

\begin{Remark}
  \label{rem:simply-connected-assumption}
This is the first time we make use of the simply-connectedness hypothesis, in order to have all the fundamental representations. However, this is largely not needed. Let $G'$ be a group with universal cover equal to $G$. Then $U^+$ and $U^-$ are canonically identified with the positive and negative unipotent subgroups of $G'$.  If we assume that $(\text{char } \kk, |\pi_1(G')| ) = 1$, then $T_1[[z^{-1}]]$ is identified with the corresponding group for $G'$. In particular, because $Q_i(z)$ depends only on the $a$-factor and $P_i(z)$ depends only on the $a$ and $x$-factors of a point $x a z^\mu y \in \overline{\cW}^\lambda_\mu$, we see that the functions $Q_i(z)$ and $P_i(z)$ are well-defined in this case as well.

If $|\pi_1(G')|$ is divisible by $\text{char } \kk$, then we suggest to take the Coulomb branch characterization of $\overline{\cW}^\lambda_\mu$ (Theorem  \ref{thm:bfn-slices-are-Coulomb-branches}) as a definition,  because this definition behaves well over general base rings.
\end{Remark}

\subsubsection{GKLO embedding}
\newcommand{\tcA}{\widetilde{\cA}}

For each $i \in I$, we define variables $w_{i,1}, \ldots, w_{i,\sv_i}$ and $\su_{i,1}, \ldots, \su_{i,\sv_i}$ and form the ring: 
\begin{equation}
  \label{eq:d:5}
\tcA(\lambda-\mu) = \kk[ w_{i,r}, u_{i,r}^{\pm 1} ]_{i \in I, r \in [\sv_i]}
\end{equation}
We further define $ \tcA(\lambda - \mu)_\loc$ to be the ring obtained by inverting all polynomials of the form $w_{i,r} - w_{i,s}$ for $i \in I$ and $1\leq r \neq s \leq \sv_i$. The group $S_{\bv}$ acts on this ring by permuting variables. 

Recall from (\ref{eq:d:6}) that we denote the coordinate ring of $\overline{\cW}_\mu^\lambda$ by $\cA(\lambda,\mu)$. As a consequence of the Coulomb branch construction of the slice $\overline{\cW}_\mu^\lambda$, we have the following theorem:

\begin{Theorem}[\cite{BFN2}]
  \label{thm:GKLO-embedding-of-slices}
  There is a birational embedding
  \begin{equation}
    \label{eq:d:12}
    \cA(\lambda,\mu) \hookrightarrow \tcA(\lambda - \mu)_\loc^{S_{\bv}}
  \end{equation}
  of algebras under which
\begin{equation}
  \label{eq:d:13}
  Q_i(z) \mapsto  \prod_{r = 1}^{\sv_i} (z - w_{i,r}) 
\end{equation}
and 
\begin{equation}
  \label{eq:d:15}
P_i(z) \mapsto \sum_{r = 1}^{\sv_i} \left( \prod_{\substack{s = 1 \\ s \neq r}}^{\sv_i}  \frac{z - w_{i,s}}{w_{i,r} - w_{i,s}} \right) \prod_{a \in \arrowset  :\, \source(a) = i} \prod_{t = 1}^{\sv_{\target(a)}} ( w_{\target(a),t} - w_{i,r}) \su_{i,r}
\end{equation}
for all $i \in I$.
\end{Theorem}

\begin{Remark}
\label{remark: GKLO vs BFN}
The above embedding differs from that in \cite{BFN2} by a sign.  More precisely, define an involution on $\tcA(\lambda-\mu)$ by 
$$
w_{i,r} \mapsto w_{i,r}, \qquad \su_{i,r} \mapsto (-1)^{\sum_{a \in E : \source(a) = i} \sv_{\target(a)}} \su_{i,r}
$$
Composing this involution with the map \eqref{eq:d:12}, we obtain precisely the $\hbar=0$ limit of \cite[Theorem B.15]{BFN2}, cf.~\cite[Lemma B.27]{BFN2}. We make this change of signs to simplify the statement of certain other results, such as Theorem \ref{thm:compatibility-of-FMOs-with-embeddings-of-slices} below. The above involution may also be interpreted as the action of a certain element of the torus $T \subset G$, acting via the adjoint action on $\overline{\cW}_\mu^\lambda$.
\end{Remark}

The map from the theorem is a slight variation of one constructed for (quantized) zastava spaces by Gerasimov, Kharchev, Lebedev, and Oblezin \cite{GKLO}. For that reason we will call this the \emph{GKLO embedding} of $\cA(\lambda,\mu)$.  This map was also studied in \cite{KWWY},  in the case where $\mu$ is dominant. Note that because the $P_i(z)$ and $Q_i(z)$ are birational coordinates on $ \overline{\cW}^\lambda_\mu$, the map is uniquely determined by formulas \eqref{eq:d:13} and \eqref{eq:d:15}.  

Consider the subalgebra of $\cA(\lambda,\mu)$ generated by the coefficients of the $Q_i(z)$. Under the GKLO embedding, this subalgebra is identified with the subalgebra $\kk[w_{i,r}]_{i\in I, r \in [\sv_i]}^{S_\bv} \subset \tcA(\lambda-\mu)$ of symmetric polynomials.  For each $i \in I$, the coefficients of $Q_i(z)$ are identified (up to a sign) with the elementary symmetric functions in the variables $w_{i,r}$ for $r \in [\sv_i]$, by \eqref{eq:d:13}.

\subsubsection{Fundamental monopole operators}

For each tuple $\bm = (\sm_i)_{i\in I} \in \ZZ^I$ with $\mathbf{0} \leq \bm \leq \bv$, we write 
\begin{equation}
\Lambda_\bm = \Lambda^\bv_\bm =  \bigotimes_{i \in I} \kk[w_{i,r} \ : \ 1 \leq r \leq \sv_i]^{S_{\sm_i} \times S_{\sv_i - \sm_i}}
\end{equation}

for the corresponding ring of partially symmetric polynomials. 

Below we will consider sums over tuples $\Gamma = (\Gamma_i)_{i \in I}$ such that $\Gamma_i \subseteq [\sv_i]$ and $\# \Gamma_i = m_i$. For each such tuple $\Gamma$, let $\sigma = (\sigma_i)_{i\in I} \in S_{\bv}$ be an element such that $\sigma_i([\sm_i]) = \Gamma_i$. 
Given $f \in \Lambda^\bv_\bm$, define $f\vert_\Gamma = \sigma(f)$, which is an element of $\bigotimes_{i \in I} \kk[w_{i,r} \ : \ 1 \leq r \leq \sv_i]$. Because $f$ is partially symmetric, $f\vert_\Gamma$ does not depend on which $\sigma$ was chosen. Finally, we define $\su_\Gamma = \prod_{i \in I} \prod_{r \in \Gamma_i} \su_{i,r}$.

\begin{Definition}
\label{def: FMOs}
Let $\bm = (\sm_i)_{i\in I} \in \ZZ^I$ with $\mathbf{0} \leq \bm \leq \bv$, and let $f \in \Lambda^\bv_\bm$.  We define the \emph{(positive) dressed fundamental monopole operator} $\M^{+}_\bm(f) \in \tcA(\lambda - \mu)_\loc$ by the formula:
\begin{equation}
\label{eq:positive-FMOs-images-under-GKLO}
\M^{+}_\bm(f) = \sum_{\substack{\Gamma = (\Gamma_i)_{i \in I} \\ \Gamma_i \subseteq [\sv_i],  \# \Gamma_i = \sm_i}} f\vert_\Gamma \cdot  \frac{\prod_{ a \in \arrowset} \prod_{r\in \Gamma_{\source(a)}, s\notin \Gamma_{\target(a)}} (w_{\target(a),s} - w_{\source(a),r})}{\prod_{i \in I} \prod_{r \in \Gamma_i, s\notin \Gamma_i}(w_{i,r} - w_{i,s})}  \su_{\Gamma} 
\end{equation}
We define the \emph{(negative) dressed fundamental monopole operator} $\M^{-}_\bm(f) \in \tcA(\lambda - \mu)_\loc$ by the formula: 
\begin{equation}
\label{eq:negative-FMOs-images-under-GKLO}
\M^{-}_\bm(f) = (-1)^{\text{sign}}  \sum_{\substack{\Gamma = (\Gamma_i)_{i \in I} \\ \Gamma_i \subseteq [\sv_i],  \# \Gamma_i = \sm_i}} f\vert_\Gamma \cdot
\frac{\prod_{j \in I} \prod_{t \in \Gamma_j} w_{j,t}^{\sw_{j}} \cdot \prod_{ a \in \arrowset} \prod_{r\in \Gamma_{\target(a)}, s\notin \Gamma_{\source(a)}} (w_{\target(a),r} - w_{\source(a),s}  )}{\prod_{i \in I} \prod_{r \in \Gamma_i, s\notin \Gamma_i}(w_{i,s} - w_{i,r})}  \su_{\Gamma}^{-1} 
\end{equation}
where: 
\begin{align}
  \label{eq:d:48}
  \text{sign} = \sum_{i \in I} \sm_i \sv_i + \sum_{a \in E} \sm_{\source(a)} \sv_{\target(a)}
\end{align}

\end{Definition}

\begin{Remark}
When $\bm = \mathbf{0}$, meaning that $\sm_i = 0$ for all $i \in I$, we interpret 
$$
\M^\pm_{\mathbf{0}}(f) = f \ \in \ \Lambda_\mathbf{0} = \kk[w_{i,r}]_{i \in I, r \in [\sv_i]}^{S_\bv}
$$
Furthermore, for general $\bm$ we have $\Lambda_{\mathbf{0}} \subseteq \Lambda_\bm$, and a linearity property over this subring $\Lambda_{\mathbf{0}}$:
$$
M^\pm_\bm( f g) = f \cdot M^\pm_\bm(g), \qquad \text{for all } f \in \Lambda_{\mathbf{0}}, \ g \in \Lambda_\bm
$$
\end{Remark}

For brevity, we will abbreviate ``dressed fundamental monopole operator'' to \emph{FMO}. Observe that the FMOs lie in $\tcA(\lambda-\mu)_\loc^{S_\bv}$, so via Theorem \ref{thm:GKLO-embedding-of-slices} the FMOs are rational functions on $\overline{\cW}^\lambda_\mu$. The following non-trivial theorem is a consequence of the Coulomb branch construction of $\overline{\cW}^\lambda_\mu$.

\begin{Theorem}[\cite{BFN2}]
  \label{thm: FMOs are regular}
The FMOs are regular functions on $\overline{\cW}^\lambda_\mu$ 
\end{Theorem}
Additionally we have the following theorem of Weekes.

\begin{Theorem}[\cite{Weekes1}]
\label{thm: FMOs generate for slice}
  The FMOs generate the ring of regular functions on $\overline{\cW}^\lambda_\mu$.
\end{Theorem}

The above results come from the mathematical theory of Coulomb branches, which we review in \S\ref{section: KM slices via Coulomb}, and in particular rely on Braverman, Finkelberg and Nakajima's groundbreaking Theorem \ref{thm:bfn-slices-are-Coulomb-branches}. The FMOs arise very naturally in this setting, see \S \ref{sec: FMOs}.

\begin{Remark}
\label{remark: FMOs and sign}
Up to a sign, the elements $\M^\pm_\bm(f)$ do not depend on the choice of orientation of the Dynkin diagram. This is explained most naturally in the context of Coulomb branches, see \S \ref{sec: FMOs} and equation \eqref{eq: FT for MMOs} below.
\end{Remark}
\subsection{Compatibility of FMOs with inclusions of affine Grassmannian slices}
\label{sec:Compatibility of FMOs with inclusions of generalized affine Grassmannian slices}

Fix another $\lambda' \in \coweightlattice_{++}$ such that $\lambda \geq \lambda' \geq \mu$. Then there is a closed embedding $\overline{\cX^{\lambda'}} \hookrightarrow \overline{\cX^{\lambda}}$ of (preimages of) spherical Schubert varieties; this is easy to see for the spaces $\overline{\XX^\lambda}$ from \eqref{eq:d:3}. Intersecting with the space $\cW_\mu $ from \eqref{eq:d:2}, there is thus a closed embedding:
\begin{equation}
  \label{eq:d:10}
 \overline{\cW}^{\lambda'}_\mu \hookrightarrow  \overline{\cW}^\lambda_\mu
\end{equation}
Define the tuple $\bv' = (\sv'_i)_{i \in I}$ by $\sv'_i = \langle \lambda - {\lambda'}, \Lambda_i \rangle$ for each $i \in I$.

Let $\bm = (\sm_i)_{i \in I} \in \ZZ^I$ with $\mathbf{0} \leq \bm \leq \bv$. Suppose further that $\bm \leq \bv'$. Then we can define a map
\begin{equation}
  \label{eq:d:103}
 \Lambda_\bm^\bv \rightarrow  \Lambda_\bm^{\bv'} : f \mapsto \tilde{f}
\end{equation}
by setting the variables $w_{i,r}$ with $r > \sv'_i$ equal to zero for all $i \in I$. We can now state the main theorem of this section.

\begin{Theorem}
  \label{thm:compatibility-of-FMOs-with-embeddings-of-slices}

 Let $\bm = (\sm_i)_{i\in I} \in \ZZ^I$ with $\mathbf{0} \leq \bm \leq \bv $, and let $f \in \Lambda^\bv_\bm$. Under the restriction map of functions $\kk[\overline{\cW}^\lambda_\mu] \rightarrow \kk[\overline{\cW}^{\lambda'}_\mu]$ corresponding to the closed embedding \eqref{eq:d:10}, we have
 \begin{equation}
   \label{eq:d:11}
   \M_{\bm}^{\pm}(f) \mapsto
   \begin{cases}
    \M_\bm^{\pm}\big(\tilde{f}\big) & \text{if } \bm \leq \bv'  \\ 0 & \text{otherwise}  
   \end{cases}
 \end{equation}

\end{Theorem}

In the theorem and in the proof below, we assume that the FMOs on both sides are defined using the same orientation of the underlying Dynkin diagram.  But this assumption is ultimately of little importance, as the FMOs are independent of orientation up to a sign by Remark \ref{remark: FMOs and sign}.

\subsubsection{Examples}
\label{sec: examples}
Although the definition of a general FMO $\M_\bm^\pm(f)$ is given by a complicated rational expression, we recover some familiar functions as special cases:
\begin{equation}
\label{eq: ex1}
Q_i(z) = \M_{\mathbf{0}}^\pm \Big( \prod_{r=1}^{\sv_i} (z-w_{i,r})\Big),\qquad P_i(z) = \M_{e_i}^+\Big( \prod_{r=2}^{\sv_i} (z - w_{i,r}) \Big)
\end{equation}
Here $e_i \in \ZZ^I$ denotes the $i$-th standard basis vector, and we have extended the definition of $\M_\bm^\pm(f)$ to allow for $f\in \Lambda_\bm^\bv[z]$.  Equation (\ref{eq: ex1}) follows easily from Theorem \ref{thm:GKLO-embedding-of-slices} and Definition \ref{def: FMOs}.

Similarly, in \S \ref{sec: Chevalley} below we define functions $P_i^-(z)$, which are the images of $P_i(z)$ under the Chevalley involution. These functions are also related to FMOs: 
\begin{equation}
\label{eq: ex2}
P_i^-(z) = \M_{e_i}^-\Big( \prod_{r=2}^{\sv_i} (z-w_{i,r}) \Big)
\end{equation}

\begin{Remark}
Note that $Q_i(z), P_i(z)$ and $P_i^-(z)$ are all expressed in terms of representation-theoretic data: they are matrix coefficients. The same is therefore true of the corresponding FMOs in (\ref{eq: ex1}) and (\ref{eq: ex2}), which in particular confirms Theorem \ref{thm: FMOs are regular} in these cases.   It would interesting to find a representation-theoretic or group-theoretic interpretation for more general FMOs $\M_\bm^\pm(f)$. We are not aware of such an interpretation even for $G = SL_2$.
\end{Remark}

\subsection{Zastava spaces}
\label{sec: zastava spaces}

Let $\gamma \in \corootlattice_+$. Define $\bbzastavap^\gamma$ to be the following closed subscheme of $G((z^{-1}))$
\begin{align}
  \label{eq:d:9}
\bbzastavap^\gamma = \Big\{ xa z^{-\gamma} \in U_1^+[[z^{-1}]] T_1[[z^{-1}]] z^{-\gamma} \ \suchthat \ \operatorname{val} \big( x a z^{-\gamma} | v_{-\Lambda}  \rangle \big) \geq 0 \text{ for all } \Lambda \Big\}
\end{align}

where $\Lambda \in \weightlattice_{++}$ varies over all dominant weights and $| v_{-\Lambda} \rangle \in V(-\Lambda)$ is the lowest weight vector, and $V(-\Lambda)$ varies over $W(-\Lambda)$ and $S(-\Lambda)$. It is not clear to us whether $\bbzastavap^\gamma$ is a reduced scheme, so we define the \emph{zastava space} $\zastavap^{\gamma}$ to be the reduced scheme of $\bbzastavap^\gamma$.

We have a map $\zastavap^\gamma \rightarrow \AA^{(\gamma)}$ sending the point $x a z^{-\gamma}$ to the colored divisor $\big( Q_i(z) \big)_{i\in I}$ where each $Q_i(z) = z^{\langle \gamma, \fundweight_i \rangle} \langle v_{-\fundweight_i^\ast} | a | v_{- \fundweight_i} \rangle $.

\begin{Remark}
The zastava space is often defined as a space of quasimaps from a rational curve into a flag variety. 
Our definition of the zastava rather is most naturally isomorphic to the equivalent Beilinson-Drinfeld picture of the zastava (see e.g. \cite[\S 6]{Finkelberg-Mirkovic}). Given a point $x a z^{-\gamma} \in \zastavap^\gamma$, we extract a colored divisor as above; let us call it $D$. It is easy to see that $x a z^{-\gamma} \in \zastavap^\gamma$ implies that $x$ is in fact a rational function of $z$ regular away from the support of $D$.  So using $x^{-1}$, we obtain a rational trivialization on the trivial $U$-bundle on $\PP^1$, and our valuation condition exactly translates into the required pole conditions along $D$. 
\end{Remark}

We have a natural map $\overline{\cW}^\lambda_\mu \rightarrow \zastavap^{\lambda - \mu}$ given by $x a z^\mu y \mapsto x a z^{\mu - \lambda}$. Observe that the functions $P_i(z)$ and $Q_i(z)$ on $\overline{\cW}^\lambda_\mu$ are pulled back from regular functions on $\zastavap^{\lambda - \mu}$ given by the same formulas. It is known \cite{FKMM,BDF} that these functions are birational coordinates on $\zastavap^{\lambda - \mu}$. In particular, the map 
$\overline{\cW}^\lambda_\mu \rightarrow \zastavap^{\lambda - \mu}$ is birational.

 \begin{Remark}
 \label{rmk: diagram aut}
   Braverman, Finkelberg, and Nakajima \cite{BFN2} make use of a map  $\overline{\cW}^{\lambda^*}_{\mu^*} \rightarrow \zastavap^{\lambda - \mu}$, where $*$ is the diagram automorphism given by ``minus the longest element of the Weyl group''. Because $G$ is equipped with a pinning, this diagram automorphism induces an automorphism of $G$, which induces an isomorphism $\overline{\cW}^{\lambda}_{\mu} \overset{\sim}{\rightarrow} \overline{\cW}^{\lambda^*}_{\mu^*}$. Our map $\overline{\cW}^\lambda_\mu \rightarrow \zastavap^{\lambda - \mu}$ is exactly the composed map.
  \end{Remark}

\subsubsection{GKLO embedding for the zastava space} Let $\tcA^+(\lambda-\mu) = \kk[ w_{i,r}, u_{i,r}]_{i \in I, r \in [\sv_i]}$, and define $\tcA^+(\lambda-\mu)_\loc$ to be the localization at all $w_{i,r} - w_{i,s}$ for all $1 \leq r \neq s \leq \bv_i$. Note that in these rings we do not invert the elements $u_{i,r}$. Then analogously to Theorem \ref{thm:GKLO-embedding-of-slices}, we have a birational embedding
\begin{equation}
  \label{eq:d:7}
  \kk[\zastavap^{\lambda - \mu}] \hookrightarrow \tcA^+(\lambda-\mu)_\loc
\end{equation}
sending $Q_i(z)$ and $P_i(z)$ to the elements given by formulas \eqref{eq:d:13} and \eqref{eq:d:15}.   Similarly to \eqref{eq:d:6}, we will denote $\cA^+(\lambda-\mu) =   \kk[\zastavap^{\lambda - \mu}]$.

In particular, we have the following commutative square of birational algebra embeddings:
\begin{equation}
  \label{eq:d:8}
\begin{tikzcd}
\kk[\zastavap^{\lambda - \mu}]   \arrow[d,hook] \arrow[r, hook] & \tcA^+(\lambda-\mu)_\loc  \arrow[d,hook] \\
\kk[\overline{\cW}^\lambda_\mu]   \arrow[r,hook]                        & \tcA(\lambda-\mu)_\loc 
\end{tikzcd}
\end{equation}
Observe that the positive FMOs define rational functions on $\zastavap^{\lambda - \mu}$.  In fact they are regular functions, a consequence of the construction of the zastava space in Coulomb branch terms \cite[\S 3(ii)]{BFN2}, see also \S \ref{section: zastava spaces via Coulomb} below.  An analogue of Theorem \ref{thm: FMOs generate for slice} also holds: the positive FMOs $\M_\bm^+(f)$ generate the ring of regular functions on $\zastavap^{\lambda - \mu}$ \cite{Weekes2}.

\subsubsection{Compatibility with closed embeddings}

If $\gamma' \leq  \gamma$ are elements of $\corootlattice_+$, then we have a closed embedding $\zastavap^{\gamma'} \hookrightarrow \zastavap^\gamma$ given by $xa z^{-\gamma'} \mapsto x a z^{-\gamma}$, which is compatible with the closed embedding of  slices. We state this as the following proposition. 

\begin{Proposition}
Suppose $\lambda, \lambda' \in \coweightlattice_{++}$ and $\mu \in \coweightlattice$ with $\lambda \geq \lambda' \geq \mu$. Then the following square commutes. 
  \begin{equation}
    \label{eq:8}
\begin{tikzcd}
\overline{\cW}^{\lambda'}_\mu \arrow[d] \arrow[r, hook] & \overline{\cW}^{\lambda}_\mu  \arrow[d] \\
\zastavap^{\lambda' - \mu}  \arrow[r,hook]                        & \zastavap^{\lambda - \mu}          
\end{tikzcd}
  \end{equation}
\end{Proposition}

We see then that for positive FMOs, Theorem \ref{thm:compatibility-of-FMOs-with-embeddings-of-slices} reduces to the following analogous theorem for zastava spaces.
\begin{Theorem}
  \label{thm:compatibility-of-FMOs-with-embeddings-of-zastava}
  Let $\gamma', \gamma \in \corootlattice_+$ with $\gamma' \leq \gamma$. Write $\gamma = \sum_{i \in I} \sv_i \simplecoroot_i$ and $\gamma' = \sum_{i \in I} \sv_i' \simplecoroot_i$. Consider the closed embedding $\zastavap^{\gamma'} \hookrightarrow \zastavap^\gamma$ and the corresponding restriction map $\kk[\zastavap^\gamma] \rightarrow \kk[\zastavap^{\gamma'}]$ of functions. Let $\bm = (\sm_i)_{i\in I} \in \ZZ^I$ with $\mathbf{0} \leq \bm \leq \bv$, and let $f \in \Lambda_\bm$. Then under the restriction map we have
 \begin{equation}
   \label{eq:d:11}
   \M_{\bm}^{+}(f) \mapsto
   \begin{cases}
    \M_\bm^{+}(\tilde{f}) & \text{if } \bm \leq \bv'  \\ 0 & \text{otherwise}  
   \end{cases}
 \end{equation}
\end{Theorem}

The difficulty with proving Theorem \ref{thm:compatibility-of-FMOs-with-embeddings-of-slices} is that the FMOs are expressed as rational expressions, and the image of the closed embedding 
$ \overline{\cW}^{\lambda'}_\mu \hookrightarrow  \overline{\cW}^\lambda_\mu$
 lies in the vanishing locus of the denominators of those rational expressions. Therefore, it is not clear how to restrict these functions to the closed subset. Unfortunately, Theorem \ref{thm:compatibility-of-FMOs-with-embeddings-of-zastava} also suffers from this difficulty.

Instead, our strategy is to prove Theorem \ref{thm:compatibility-of-FMOs-with-embeddings-of-zastava} by proving a stronger statement (Theorem \ref{thm:compatibility-of-FMOs-with-adding-defect-map-of-zastava}) for the ``adding defect'' map of zastava spaces that we will recall below. The image of the adding defect map is no longer contained in the vanishing locus of the denominators of the rational expressions defining the FMOs, so we can restrict the FMOs.

\subsection{Adding defect}
\label{sec: adding defect}
Let $\gamma'' \in \corootlattice_+$. Observe that $\AA^{(\gamma'')}$ is isomorphic to the variety of $a'' \in T_1[[z^{-1}]]$ such that $ z^{\langle \gamma'', \fundweight_i\rangle} \langle v_{-\fundweight_i}^* | a'' | v_{-\fundweight_i} \rangle$ is a polynomial in $z$ for all $i \in I$. The isomorphism sends $a''$ to the colored divisor $(L_i(z))_{i \in I}$, where $L_i(z) = z^{\langle \gamma'', \fundweight_i\rangle} \langle v_{-\fundweight_i}^* | a'' | v_{-\fundweight_i} \rangle$.

\begin{Definition}
  \label{def:adding-defect}
  Let $\gamma', \gamma'' \in \corootlattice_+$ and let $\gamma = \gamma' + \gamma''$. The \emph{adding-defect} map
  \begin{equation}
    \label{eq:d:14}
     \zastavap^{\gamma'} \times \AA^{(\gamma'')} \rightarrow \zastavap^\gamma
  \end{equation}
  is defined by
  \begin{equation}
    \label{eq:d:25}
    (x'a'z^{-\gamma'}, a'') \mapsto x' a' a'' z^{-\gamma}
  \end{equation}
\end{Definition}

Fix $\gamma',\gamma''$, and $\gamma$ as in the definition. Recall the regular functions $Q_i(z)$ and $P_i(z)$ on $\zastavap^\gamma$. We will write $\overline{Q_i}(z)$ and $\overline{P_i}(z)$ for the corresponding functions on $\zastavap^{\gamma'}$.  As above, we have regular functions $L_i(z)$  on $\AA^{(\gamma'')}$. It follows directly from the formula defining the adding defect map that
\begin{equation}
  \label{eq:d:16}
  Q_i(z) \mapsto \overline{Q_i}(z) L_i(z) 
\end{equation}
and:
\begin{equation}
  \label{eq:d:17}
  P_i(z) \mapsto \overline{P_i}(z) L_i(z) 
\end{equation}
Furthermore, because $Q_i(z)$ and $P_i(z)$ are birational coordinates, we see that these formulas uniquely determine the map.

\subsubsection{GKLO and adding defect}

Define $\bv' = (\sv'_i)_i, \bv'' = (\sv''_i)_i$, and $\bv = (\sv_i)_i$ by $\gamma^\bullet = \sum_{i \in I} \sv^\bullet_i \simplecoroot_i$ where $\bullet \in \{','',\emptyset\}$.

Define $\Lambda^{(\bv',\bv]} = \kk[w_{i,r}]_{i \in I, r \in \{\sv_i'+1, \ldots, \sv_i\}}$. In addition to the GKLO embeddings for the zastava space, we have an embedding
\begin{equation}
  \label{eq:d:18}
 \kk[ \AA^{(\gamma'')} ] \hookrightarrow \Lambda^{(\bv',\bv]}
\end{equation}
given by: 
\begin{equation}
  \label{eq:d:19}
  L_i(z) \mapsto \prod_{r = \sv_i'+ 1}^{\sv_i} (z - w_{i,r}) 
\end{equation}
Next, define
\begin{equation}
  \label{eq:d:20}
  \left(\tcA^+(\gamma') \otimes \Lambda^{(\bv',\bv]} \right)_{\loc}
\end{equation}
to be the localization given by inverting $w_{i,r} - w_{j,s}$ for all pairs $(i,r) \neq (j,s)$.  The GKLO embedding and \eqref{eq:d:18} induce an embedding:
\begin{equation}
  \label{eq:d:21}
  \kk[ \zastavap^{\gamma'} \times \AA^{(\gamma'')}] \hookrightarrow \left(\tcA^+(\gamma') \otimes \Lambda^{(\bv',\bv]}   \right)_{\loc}
\end{equation}

Finally, we define a map
\begin{equation}
  \label{eq:d:22}
  \phi : \tcA^+(\gamma)_\loc \rightarrow \left(\tcA^+(\gamma') \otimes \Lambda^{(\bv',\bv]}  \right)_{\loc} 
\end{equation}
by sending $w_{i,r} \mapsto w_{i,r}$ and by sending:
\begin{equation}
  \label{eq:d:23}
  \su_{i,r} \mapsto
  \begin{cases}
    \frac{ \prod_{s = \sv_i' + 1}^{\sv_i} (w_{i,r} - w_{i,s})}{\prod_{a \in E :\, \source(a) = i} \prod_{t = \sv'_{\target(a)}+1}^{\sv_{\target(a)}} (w_{\target(a),t} - w_{i,r})} \su_{i,r}, & \text{ if } r \leq \sv_i' \\
    0, & \text{otherwise} 
  \end{cases}
\end{equation}
for all $i \in I$ and $1\leq r \leq \sv_i$.

\begin{Theorem}
  \label{thm:GKLO-commutative-square-for-adding-defect}
Let $\gamma',\gamma''$, and $\gamma$ be as in Definition \ref{def:adding-defect}, and let $\bv',\bv'',\bv$ be as above. The following diagram commutes:
\begin{equation}
  \label{eq:d:24}
  \begin{tikzcd}
    \kk[\zastavap^\gamma]   \arrow[d] \arrow[r, hook] & \tcA^+(\gamma)_\loc  \arrow[d,"\phi"] \\
\kk[ \zastavap^{\gamma'} \times \AA^{(\gamma'')}]   \arrow[r,hook]                        & 
\left(\tcA^+(\gamma') \otimes \Lambda^{(\bv',\bv]} \right)_{\loc}
  \end{tikzcd}
\end{equation}
\end{Theorem}
\begin{proof}
It suffices to check that $Q_i(z)$ and $P_i(z)$ map according to \eqref{eq:d:16} and \eqref{eq:d:17} under the top and right arrows. This is immediate for $Q_i(z)$, and \eqref{eq:d:23} is exactly what makes it hold for $P_i(z)$. 

\end{proof}

\subsubsection{FMOs and adding defect}

Fix $\bm = (\sm_i)_{i\in I} \in \ZZ^I$ such that $\mathbf{0} \leq \bm \leq \bv$. We have a embedding:
\begin{equation}
  \label{eq:d:26}
  \Lambda^\bv_\bm \hookrightarrow \Lambda^{\bv'}_\bm \otimes \Lambda^{(\bv',\bv]}
\end{equation}
Given $f \in \Lambda^\bv_\bm$, we use Sweedler notation and write
\begin{equation}
  \label{eq:d:27}
  f = \sum f^{(1)} \otimes f^{(2)} 
\end{equation}
for the image of $f$ under \eqref{eq:d:26}. 

\begin{Theorem}
  \label{thm:compatibility-of-FMOs-with-adding-defect-map-of-zastava}
  Let $\gamma',\gamma''$, and $\gamma$ be as in Definition \ref{def:adding-defect}, and let $\bv',\bv'',\bv$ be as above.   Let $\bm = (\sm_i)_{i\in I} \in \ZZ^I$ with $\mathbf{0} \leq \bm \leq \bv$, and let $f \in \Lambda^\bv_\bm$. Write $f = \sum f^{(1)} \otimes f^{(2)}$ under the embedding \eqref{eq:d:26}.
Under the restriction map on coordinate rings corresponding to the adding defect map \eqref{eq:d:14} we have:
 \begin{equation}
   \label{eq:d:11}
   \M_{\bm}^{+}(f) \mapsto
   \begin{cases}
    \sum \M_\bm^{+}(f^{(1)}) \otimes f^{(2)}, & \text{if } \sm_i \leq \sv'_i \text{ for all } i \in I, \\ 0, & \text{otherwise}  
   \end{cases}
 \end{equation}
\end{Theorem}

Observe that the closed embedding $ \zastavap^{\gamma'}   \hookrightarrow \zastavap^\gamma$ factors through the adding defect map as $ \zastavap^{\gamma'} \hookrightarrow \zastavap^{\gamma'} \times \AA^{(\gamma'')} \rightarrow \zastavap^\gamma$ where the closed embedding $ \zastavap^{\gamma'} \hookrightarrow \zastavap^{\gamma'} \times \AA^{(\gamma'')}$ is given by the closed point of $\AA^{(\gamma'')}$ corresponding to the unique colored divisor supported at $0 \in \AA^1$. Observe also that given $f \in \Lambda^\bv_\bm$ expressed as $\sum f^{(1)} \otimes f^{(2)}$ under \eqref{eq:d:26}, we have $\tilde{f} = \sum f^{(1)} f^{(2)}(0)$, where $f^{(2)}(0)$ is obtained by setting all variables $w_{i,r}$ with $r > \sv'_i$ equal to zero. We therefore obtain Theorem \ref{thm:compatibility-of-FMOs-with-embeddings-of-zastava} as a corollary of Theorem \ref{thm:compatibility-of-FMOs-with-adding-defect-map-of-zastava}. As explained above, we therefore obtain Theorem \ref{thm:compatibility-of-FMOs-with-embeddings-of-slices} for \emph{positive} FMOs.

\begin{proof}[Proof of Theorem \ref{thm:compatibility-of-FMOs-with-adding-defect-map-of-zastava}]
  Recall the rational expression \eqref{eq:positive-FMOs-images-under-GKLO} for $M_\bm^+(f)$. By Theorem \ref{thm:GKLO-commutative-square-for-adding-defect}, we are reduced to applying the map $\phi$ to this rational expression. Because all $\su_{i,r}$ with $r > \sv_i'$ map to zero under $\phi$ in the sum  \eqref{eq:positive-FMOs-images-under-GKLO}, we need only consider $\Gamma = (\Gamma_i)_{i\in I}$ such that $\Gamma_i \subseteq [\sv_i']$. That is, we are computing:
  \begin{equation}
    \label{eq:d:28}
\phi \left(    
\sum_{\substack{\Gamma = (\Gamma_i)_{i \in I} \\ \Gamma_i \subseteq [\sv'_i] \subseteq [\sv_i],  \# \Gamma_i = \sm_i}} f\vert_\Gamma \cdot  \frac{\prod_{ a \in \arrowset} \prod_{r\in \Gamma_{\source(a)}, s\notin \Gamma_{\target(a)}} (w_{\target(a),s} - w_{\source(a),r})}{\prod_{i \in I} \prod_{r \in \Gamma_i, s\notin \Gamma_i}(w_{i,r} - w_{i,s})}  \su_{\Gamma} 
\right)
  \end{equation}
If $\sm_i > \sv_i'$ for any $i$, we immediately see that the sum is empty, and we get $0$. So we may assume $\sm_i \leq \sv_i'$ for all $i$.

Consider a single term in the sum \eqref{eq:d:28}, i.e. fix a $\Gamma = (\Gamma_i)_{i \in I}$ such that $\Gamma_i \subseteq [\sv'_i] \subseteq [\sv_i]$  and $\# \Gamma_i = \sm_i$ for all $i$. First, observe that:
\begin{equation}
  \label{eq:d:29}
  f\vert_\Gamma = \sum f^{(1)}\vert_\Gamma \cdot f^{(2)}
\end{equation}
In the numerator of  our single term, we write
\begin{align}
  \label{eq:d:30}
  &\prod_{ a \in \arrowset} \prod_{r\in \Gamma_{\source(a)}, s\notin \Gamma_{\target(a)}} (w_{\target(a),s} - w_{\source(a),r}) = \\
  &\left(\prod_{ a \in \arrowset} \prod_{r\in \Gamma_{\source(a)}, s\in [\sv_{\target(a)}'] \backslash \Gamma_{\target(a)}} (w_{\target(a),s} - w_{\source(a),r}) \right)  \cdot \left(\prod_{ a \in \arrowset} \prod_{r\in \Gamma_{\source(a)}} \prod_{s = \sv_{\target(a)}' + 1}^{\sv_{\target(a)}}  (w_{\target(a),s} - w_{\source(a),r}) \right)
\end{align} 
and in the denominator, we similarly write:
\begin{equation}
  \label{eq:d:31}
\prod_{i \in I} \prod_{r \in \Gamma_i, s\in \Gamma_i}(w_{i,r} - w_{i,s}) =   
\left( \prod_{i \in I} \prod_{r \in \Gamma_i, s\in [\sv_i'] \backslash \Gamma_i}(w_{i,r} - w_{i,s}) \right) \cdot
\left( \prod_{i \in I} \prod_{r \in \Gamma_i} \prod_{s = \sv_i'+1}^{\sv_i}(w_{i,r} - w_{i,s}) \right)
\end{equation}
Finally, observing that
\begin{equation}
  \label{eq:d:32}
  \phi\left( \frac{\prod_{ a \in \arrowset} \prod_{r\in \Gamma_{\source(a)}} \prod_{s = \sv_{\target(a)}' + 1}^{\sv_{\target(a)}}  (w_{\target(a),s} - w_{\source(a),r})}{ \prod_{i \in I} \prod_{r \in \Gamma_i} \prod_{s = \sv_i'+1}^{\sv_i}(w_{i,r} - w_{i,s}) } u_{\Gamma} \right)  = u_{\Gamma}
\end{equation}
we see that \eqref{eq:d:28} is equal to 
\begin{equation}
  \label{eq:d:33}
\sum \left( \sum_{\substack{\Gamma = (\Gamma_i)_{i \in I} \\ \Gamma_i \subseteq [\sv'_i],  \# \Gamma_i = \sm_i}} f^{(1)}\vert_\Gamma \cdot  \frac{\prod_{ a \in \arrowset} \prod_{r\in \Gamma_{\source(a)}, s\notin \Gamma_{\target(a)}} (w_{\target(a),s} - w_{\source(a),r})}{\prod_{i \in I} \prod_{r \in \Gamma_i, s\notin \Gamma_i}(w_{i,r} - w_{i,s})}  \su_{\Gamma} \right) \cdot f^{(2)}
\end{equation}
which is exactly the rational expression for $\sum M_\bm^{+}(f^{(1)}) f^{(2)}$. 
\end{proof}

\subsection{Chevalley anti-involution on slices}
\label{sec: Chevalley}

In this section, we will recall an involution on the slice $\overline{\cW}^\lambda_\mu$ that swaps  positive and negative FMOs and allows us to finish the proof of Theorem \ref{thm:compatibility-of-FMOs-with-embeddings-of-slices} (by proving it for the negative FMOs). This involution and its non-commutative version for Yangians is studied in detailed in \cite[\S 2(vii) and Appendix B]{BFN2}, and the statements of this section can be obtained there. We present the key facts necessary for the present commutative situation.

As the group $G$ is equipped with a pinning, we have the Chevalley anti-involution $\iota: G \rightarrow G$ obtained by swapping positive and negative root subgroups. We get a corresponding anti-involution on the group $G((z^{-1}))$ that preserves the slice $\overline{\cW}^\lambda_\mu$. We denote this involution by the same symbol $\iota: \overline{\cW}^\lambda_\mu \rightarrow \overline{\cW}^\lambda_\mu$. We also write $\iota : \kk[\overline{\cW}^\lambda_\mu] \rightarrow \kk[\overline{\cW}^\lambda_\mu]$ for the corresponding involution of coordinate rings. Write $P_i^{+}(z) = P_i(z)$ and define the regular function $P_i^{-}(z)$ on $\overline{\cW}^\lambda_\mu$ by
\begin{equation}
  \label{eq:d:34}
  P_i^{-}(z) = z^{\langle \lambda,\fundweight_i\rangle}\langle v_{-\fundweight_i}^* | xaz^{\mu} y | v_{-s_i(\fundweight_i)} \rangle 
\end{equation}
for $xaz^{\mu} y \in \overline{\cW}^\lambda_\mu$. Then we have $\iota(Q_i(z)) = Q_i(z)$ and $\iota(P_i^{+}(z)) = P_i^{-}(z)$.  Because the $Q_i(z)$ and $P^{+}_i(z)$ are birational coordinates, this uniquely determines $\iota$.  In order to study how this involution interacts with the GKLO embeddings, we need to compute the image of $P_i^{-}(z)$ under the GKLO embedding.

\begin{Proposition}
  Under the GKLO embedding \eqref{eq:d:12},
  \begin{equation}
    \label{eq:d:35}
    P_i^{-}(z) \mapsto - (-1)^{\sum_{b\in E : \source(b) = i } \sv_{\target(b)}} \sum_{r = 1}^{\sv_i} \left( \prod_{\substack{s = 1 \\ s \neq r}}^{\sv_i}  \frac{z - w_{i,s}}{w_{i,r} - w_{i,s}} \right) w_{i,r}^{\sw_i} \prod_{a \in E : \target(a) = i} \prod_{s = 1}^{\sv_{\source(a)}} (w_{i,r} - w_{\source(a),s}) u_{i,r}^{-1} 
  \end{equation}
\end{Proposition}

As in Remark \ref{remark: GKLO vs BFN}, this result agrees with the $\hbar = 0$ limit of \cite[Theorem B.15]{BFN2}, up to the corresponding sign.

\begin{proof}
  Observe that \eqref{eq:d:35} is in the form of the Lagrange interpolation formula, so it suffices to compute $P_i^{-}(w_{i,r})$ for all $r \in [\sv_i]$. 
  Define $D_i(z)$ to be the regular function on $\overline{\cW}^\lambda_\mu$ given by
  \begin{equation}
    \label{eq:d:37}
   D_i(z) =  z^{\langle \lambda,\fundweight_i\rangle}\langle v_{-s_i(\fundweight_i)}^* | xaz^{\mu} y | v_{-s_i(\fundweight_i)} \rangle 
 \end{equation}
 for $x a z^{\mu} y \in \overline{\cW}^\lambda_\mu$.
Then we have the following identity (see e.g.~\cite[Theorem 1.17]{FZ}, \cite[(2.19)]{GKLO}):
\begin{equation}
  \label{eq:d:38}
  D_i(z) Q_i(z) = P_i^{+}(z) P_i^{-}(z) + z^{\langle \lambda, \simpleroot_i \rangle} \prod_{a \in E : \target(a) = i} Q_{\source(a)}(z) \prod_{b \in E : \source(a) = i} Q_{\target(a)}(z)
\end{equation}
For each $r \in [\sv_i]$, we substitute $z = w_{i,r}$ and solve for $P_i^{-}(w_{i,r})$, noting that $Q_i(w_{i,r}) = 0$.

\end{proof}

Let $\tcA(\lambda - \mu)_{\loc,\loc}$ denote the further localization of $\tcA(\lambda - \mu)_{\loc}$ where we invert $w_{i,r}$ for all $i \in I$, $r \in [\sv_i]$. Then we can define an automorphism
\begin{equation}
  \label{eq:d:41}
 \iota: \tcA(\lambda - \mu)_{\loc,\loc} \overset{\sim}{\rightarrow}  \tcA(\lambda - \mu)_{\loc,\loc}
\end{equation}
by
\begin{equation}
  \label{eq:d:42}
  w_{i,r} \mapsto w_{i,r}
\end{equation}
and: 
\begin{align}
  \label{eq:d:43}
  \su_{i,r} & \ \mapsto \ \frac{P_i^{-}(w_{i,r})}{(-1)^{\sum_{b\in E : \source(b) = i } \sv_{\target(b)}}\prod_{b \in E : \source(b) = i} Q_{\target(b)}(w_{i,r})}  \\
& \ =  - (-1)^{\sum_{b\in E : \source(b) = i } \sv_{\target(b)}}w_{i,r}^{\sw_i} \frac{
 \prod_{a \in E : \target(a) = i} \prod_{s = 1}^{\sv_{\source(a)}} (w_{i,r} - w_{\source(a),s})      
  }{ \prod_{b \in E : \source(b) = i} \prod_{t = 1}^{\sv_{\target(b)}} (w_{\target(b),s} - w_{i,r} )} \su_{i,r}^{-1}
\end{align}
This automorphism is involutive and it has been constructed precisely to make the following proposition hold.

\begin{Proposition}
  \label{prop:involution-and-GKLO}
  The following square commutes.
  \begin{equation}
    \label{eq:d:44}
\begin{tikzcd}
\kk[\overline{\cW}^\lambda_\mu]   \arrow[d,"\iota"] \arrow[r, hook] & \tcA(\lambda-\mu)_{\loc,\loc}  \arrow[d,"\iota"] \\
\kk[\overline{\cW}^\lambda_\mu]   \arrow[r,hook]                        & \tcA(\lambda-\mu)_{\loc,\loc}
\end{tikzcd}
  \end{equation}
\end{Proposition}

\subsubsection{Negative FMOs}

With Proposition \ref{prop:involution-and-GKLO} in hand, we can directly calculate the following
\begin{Proposition}
  Fix $\bm = (\sm_i)_{i \in I}$ with $\mathbf{0} \leq \bm \leq \bv$ and $f \in \Lambda^\bv_\bm$. Under the involution $\iota:    \kk[\overline{\cW}^\lambda_\mu] \rightarrow \kk[\overline{\cW}^\lambda_\mu]$, we have
  \begin{align}
    \label{eq:d:45}
    \iota \left( M_\bm^{+}(f) \right) = M_\bm^{-}(f)
  \end{align}
\end{Proposition}
As the involution $\iota$ is compatible with closed embeddings \eqref{eq:d:10}, we obtain Theorem \ref{thm:compatibility-of-FMOs-with-embeddings-of-slices} for negative FMOs and therefore complete its proof.

\subsection{The other zastava space}

Let $\gamma \in \corootlattice_+$. Define $\bbzastavam^\gamma$ to be the following closed subscheme of $G((z^{-1}))$:
\begin{align}
  \label{eq:d:9}
\bbzastavam^\gamma = \left\{ a z^{-\gamma}y \in  T_1[[z^{-1}]] z^{-\gamma} U_1^-[[z^{-1}]] \suchthat   a z^{-\gamma}y |v_{-\Lambda}\rangle \geq 0 \text{ for all } \Lambda \right\}
\end{align}
where $\Lambda \in \weightlattice_{++}$ varies over all dominant weights and $|v_{-\Lambda}\rangle \in V(-\Lambda)$ is the lowest weight vector defined above, and $V(-\Lambda)$ varies over $W(-\Lambda)$ and $S(-\Lambda)$. As before, define $\zastavam^{\gamma}$ to be the reduced scheme of $\bbzastavam^\gamma$. There is a map $\overline{\cW}_\mu^\lambda \rightarrow \zastavam^{\lambda-\mu}$ given by $x a z^\mu y \mapsto a z^{\mu-\lambda} y$.

The involution $\iota$ induces an isomorphism:
\begin{align}
  \label{eq:d:46}
  \iota : \zastavap^\gamma \overset{\sim}{\rightarrow} \zastavam^{\gamma}
\end{align}
Applying $\iota$, all of the discussion above about $\zastavap^\gamma$ holds for $\zastavam^{\gamma}$ as well. 

Furthermore, we obtain a closed embedding:
\begin{align}
  \label{eq:d:47}
  \overline{\cW}^\lambda_\mu \hookrightarrow  \zastavap^\gamma \times_{\AA^{(\lambda - \mu)}} \zastavam^{\gamma}
\end{align}
which is given explicitly on points by:
\begin{equation}
x a z^\mu y \ \mapsto \ (x a z^{\mu-\lambda}, a z^{\mu-\lambda} y)
\end{equation}
We can interpret this embedding in terms of FMOs: the positive FMOs $\M_\bm^+(f)$ generate the coordinate ring of $\zastavap^\gamma$, the negative FMOs $\M_\bm^-(f)$ generate the coordinate ring of $\zastavam^{\gamma}$, and the map on coordinate rings corresponding to \eqref{eq:d:47} is surjective because the coordinate ring of $\overline{\cW}^\lambda_\mu$ is generated by all the FMOs $\M_\bm^\pm(f)$ together.


%% file: beyondfinite.tex
\section{Kac-Moody affine Grassmannian slices via Coulomb branches}
\label{section: KM slices via Coulomb}

In this section we explain how the Braverman-Finkelberg-Nakajima theory of Coulomb branches gives a definition of affine Grassmannian slices for all \emph{symmetric} Kac-Moody types. First we will discuss Coulomb branches in general. We then focus on the Coulomb branches of quiver gauge theories, which are those which produce Kac-Moody affine Grassmannian slices.

\subsection{Coulomb branches in general} 
\label{subsection: Coulomb branches in general}

We first recall the general definition of Coulomb branches due to Braverman, Finkelberg and Nakajima \cite{BFN1}. Let $\bG$ be a split reductive group, and let $\bN$ be a finite dimensional representation of $\bG$, both defined over $\CC$. (Note that $\bG$ will not be the same as the group $G$ from the previous section. For us $\bG$ will always be a product of general linear groups.)

We will refer to \cite[\S 3, Definition 3.13]{BFN1} for the precise definitions and here only recall the basic ingredients in the definition. Let $\cK = \CC ((t))$ and $\cO = \CC[[t]]$, and define the ind-scheme $\cR_{\bG,\bN}$ (see \cite[\S 2(i)]{BFN1}) by:
\begin{equation}
  \label{eq:d:50}
  \cR_{\bG,\bN} = \Big\{ (g,n) \in \bG(\cK)/\bG(\cO) \times \bN(\cO) \suchthat g n \in \bN(\cO) \Big\}
\end{equation}
Then define
\begin{equation}
  \label{eq:d:51}
  \cA(\bG, \bN) = H_\bullet^{\bG(\cO)} (\cR_{\bG,\bN}) 
\end{equation}
where $H_\bullet^{\bG(\cO)}$ denotes equivariant Borel-Moore homology (with coefficients in $\kk$). Note that $\cR_{\bG,\bN}$ is not an inductive limit of finite type schemes unless $\bN = 0$, so the definition of its equivariant Borel-Moore homology requires some care, see \cite[\S 2(ii)]{BFN1}.

Observe that there is a natural map $\cR_{\bG,\bN} \rightarrow \bG(\cK) / \bG(\cO)$ to the affine Grassmannian for $\bG$. Using a variation of the convolution structure on this affine Grassmannian, Braverman, Finkelberg, and Nakajima define a multiplication on $H_\bullet^{\bG(\cO)} (\cR_{\bG,\bN}) $. By formal properties of convolution algebras, this multiplication is associative \cite[Theorem 3.10]{BFN1}. However, a special feature of this setting is that the multiplication is in fact commutative  (\cite[Proposition 5.15]{BFN1}).  One can thus define the Coulomb branch $\cM_C(\bG,\bN)$ as
\begin{equation}
  \label{eq:d:49}
\cM_C(\bG,\bN) = \Spec \cA (\bG,\bN) 
\end{equation}
Furthermore, Braverman, Finkelberg and Nakajima prove that $\cA (\bG,\bN)$ is a finitely-generated $\kk$-domain, which tells us that $\cM_C(\bG,\bN)$ is an irreducible $\kk$-variety \cite[Corollary 5.22 and Proposition 6.8]{BFN1}.

\subsubsection{Equivariant cohomology of a point}

An important property of Coulomb branches is that there is a map of algebras \cite[\S 3(vi)]{BFN1}:
\begin{equation}
\label{eq: eqcohpt}
H_\bG^\bullet(pt)  \ \hooklongrightarrow \ \cA(\bG,\bN)
\end{equation}
Moreover, multiplication by this subalgebra agrees with the natural geometric action of $ H_\bG^\bullet(pt) = H_{\bG(\cO)}^\bullet(pt)$ on the equivariant Borel-Moore homology $  \cA(\bG, \bN) = H_\bullet^{\bG(\cO)} (\cR_{\bG,\bN}) $. 

\subsubsection{Localization}
\label{sec: localization}

Fix a pair of opposite Borel subgroups of $\bG$, which determines a maximal torus $\bT$ and a set $\Delta_{\bG}^{+}$ of positive roots. Then using the localization theorem in equivariant Borel-Moore homology, we have an embedding \cite[Lemma 5.10, Remark 5.23]{BFN1}  
\begin{equation}
\label{eq: localization}
\iota_\ast^{-1}: \cA(\bG,\bN) \ \hookrightarrow \ \cA(\bT,\bN)\left[  \beta^{-1}\right]_{\beta \in \Delta_{\bg}^+}
\end{equation}
where we view the roots $\beta \in \Delta_{\bG}^+$ as elements of $H^\bullet_\bT(pt)$. In our main case of interest, the quiver gauge theories to be recalled below, we will give a precise formula for this embedding in terms of the GKLO embedding (see \S \ref{sec:beyond-finite-gklo-embedding} below).

The algebra $\cA(\bT,\bN)$ is relatively simple to describe: it is free as a module over its subalgebra $H^\bullet_{\bT}(pt)$, with a basis $\{r_\gamma\}_\gamma$ indexed by the set of all coweights $\gamma:\mathbb{G}_m \rightarrow \bT$.  Moreover there is an explicit multiplication formula for these elements \cite[Theorem 4.1]{BFN1}.

\subsubsection{Forgetting matter}
Suppose that $\bNone$ and $\bNtwo$ are two $\bG$-representations. Then we have the \emph{forgetting matter} homomorphism \cite[Remark 5.14]{BFN1}:
\begin{equation}
  \label{eq:d:65}
  \bz^{*} : \cA(\bG, \bNone \oplus \bNtwo) \hookrightarrow \cA(\bG, \bNone) 
\end{equation}

To define this map, first consider the case where $\bG = \bT$ is a torus.  In this case, the map $\bz^\ast : \cA(\bT, \bNone \oplus \bNtwo) \hookrightarrow \cA(\bT, \bNone)$ is $H^\bullet_{\bT}(\pt)$-linear and for all cocharacters $\gamma: \GG_m \rightarrow \bT$ we have
\begin{equation}
  \label{eq:d:67}
  \bz^{*}(r_\gamma) = \prod_{\xi : \langle \xi, \gamma \rangle < 0 } \xi^{- \langle \xi, \gamma \rangle  \dim \bN_2(\xi) } r_{\gamma}
\end{equation}
Here the product runs over weights $\xi$ of $\bN$, thought of as elements of $H_\bT^\bullet(pt)$ \cite[\S 4(vi)]{BFN1}.

For general $\bG$, let $\bT\subset \bG$ be a maximal torus.  Then the map (\ref{eq:d:65}) is determined by the following commutative square involving the respective localization maps:
\begin{equation}
  \label{eq:d:66}
  \begin{tikzcd}
\cA(\bG,\bNone \oplus \bNtwo) \ar[r,hook,"\bz^\ast"] \ar[d,hook,"\iota_\ast^{-1}"] & \cA(\bG,\bNone) \ar[d,hook,"\iota_\ast^{-1}"] \\
\cA(\bT,\bNone\oplus \bNtwo)\left[  \beta^{-1}\right]_{\beta \in \Delta_{\bG}^+} \ar[r,hook,"\bz^\ast"] & \cA(\bT,\bNone)\left[ \beta^{-1}\right]_{\beta \in \Delta_{\bG}^+}
  \end{tikzcd}
\end{equation}
In particular, the map (\ref{eq:d:65}) is linear over $H_\bG^\bullet(pt)$.

\subsubsection{Fourier transform}
\label{subsec: FT}

The \emph{Fourier transform} is an isomorphism \cite[\S 6(viii)]{BFN1}:
\begin{equation}
  \label{eq:d:68}
  \cF : \cA(\bG,\bNone \oplus \bNtwo) \ \stackrel{\sim}{\longrightarrow} \ \cA(\bG, \bNone \oplus\bNtwo^\ast)
\end{equation}
This map is again compatible with localization:
\begin{equation}
\label{eq: Fourier via loc}
\begin{tikzcd}
\cA(\bG,\bNone \oplus \bNtwo) \ar[r,"\cF","\sim"'] \ar[d,hook,"\iota_\ast^{-1}"] & \cA(\bG,\bNone \oplus \bNtwo^\ast) \ar[d,hook,"\iota_\ast^{-1}"] \\
\cA(\bT,\bNone\oplus \bNtwo)\left[  \beta^{-1}\right]_{\beta \in \Delta_{\bG}^+} \ar[r,"\cF","\sim"'] & \cA(\bT,\bNone\oplus \bNtwo^\ast)\left[ \beta^{-1}\right]_{\beta \in \Delta_{\bG}^+}
\end{tikzcd}
\end{equation}
The bottom arrow is determined by $\cF:\cA(\bT, \bNone \oplus \bNtwo) \stackrel{\sim}{\rightarrow} \cA(\bT,\bNone \oplus \bNtwo^\ast)$. The latter map is  $H_{\bT}^\bullet(\pt)$--linear, and  on monopole operators is given by a sign \cite[\S 4(v)]{BFN1}:
\begin{equation}
  \label{eq:Fourier sign}
\cF(r_\gamma )\ = \   (-1)^{\sum_{\xi  : \langle \xi, \gamma \rangle > 0 } \langle \xi, \gamma \rangle \dim \bNtwo(\xi)} r_\gamma,
\end{equation}
In particular, the map (\ref{eq:d:68}) is $H_{\bG}^\bullet(pt)$--linear.

\subsubsection{Minuscule monopole operators}
\newcommand{\Stab}{\mathrm{Stab}}
\label{subsection: Minuscule monopole operator}

Let $\gamma : \GG_m \rightarrow \bT$ be a cocharacter. Consider the orbit $\bG(\cO) t^\lambda \bG(\cO) / \bG(\cO) \hookrightarrow \bG(\cK)/\bG(\cO)$. We say $\gamma$ is minuscule precisely when this orbit is closed; this is equivalant to asking that the dominant translate of $\gamma$ under the Weyl group $W_{\bG}$ be minuscule in the usual sense. Let $\cR^\gamma_{\bG,\bN}$ denote the preimage of the orbit under the map $\cR_{\bG,\bN} \rightarrow \bG(\cK) / \bG(\cO)$. Because of our assumption that $\gamma$ is minuscule,
\begin{equation}
  \label{eq:d:61}
  \cR^\gamma_{\bG,\bN} \rightarrow \bG(\cO) t^\lambda \bG(\cO) / \bG(\cO)
\end{equation}
is a vector bundle (of infinite rank), and we have a closed embedding  
\begin{equation}
  \label{eq:d:59}
 \cR^\gamma_{\bG,\bN} \hookrightarrow \cR_{\bG,\bN}
\end{equation}
which induces an embedding:
\begin{equation}
  \label{eq:d:60}
  H^{\bG(\cO)}_\bullet(\cR^\gamma_{\bG,\bN}) \hookrightarrow  H^{\bG(\cO)}_\bullet(\cR_{\bG,\bN})
\end{equation}
Let $\Stab_{t^{\gamma}} \subset \bG$ denote the stabilizer of $t^\lambda \bG(\cO) / \bG(\cO)$. There is an action of $H^\bullet_{\Stab_{t^{\gamma}}}(\pt)$ on $H^{\bG(\cO)}_\bullet(\cR^\gamma_{\bG,\bN})$ which is a rank-one free module with generator given by the fundamental class $[\cR^\gamma_{\bG,\bN}]$, see \cite[\S 6(ii)]{BFN1}. For $f \in H^\bullet_{\Stab_{t^{\gamma}}}(\pt)$, denote the corresponding element of $H^{\bG(\cO)}_\bullet(\cR^\gamma_{\bG,\bN})$ by
\begin{equation} 
M_\gamma^{\bG,\bN}(f) = f [\cR^\gamma_{\bG,\bN}],
\end{equation}
called the \emph{(dressed) minuscule monopole operator} (MMO) with dressing $f$.

Applying the localization theorem, we obtain the following formula (\cite[Proposition 6.6]{BFN1})
\begin{equation}
  \label{eq:d:36}
  \iota_*^{-1}\big(M_\gamma^{\bG,\bN}(f)\big) = \sum_{\sigma \in W_{\bG}/ W_{\bG,\gamma} } \sigma(f)  \cdot \frac{r_{\sigma(\gamma)}}{ \sigma \left( \prod_{\beta \in \Delta_{\bG} : \langle \beta, \gamma \rangle > 0} \beta \right) }
\end{equation}
where $W_{\bG}$ is the Weyl group of $\bG$, and $W_{\bG,\gamma}$ is the stabilizer of $\gamma$.

Finally, observe that the sign appearing in \eqref{eq:Fourier sign} is constant on $W_{\bG}$-orbits. So we see that the Fourier transform acts on MMOs by a sign: with notation as in \S \ref{subsec: FT}, 
\begin{equation}
\label{eq: FT for MMOs}
\cF\big( \M_\gamma^{\bG, \bNone \oplus \bNtwo}(f) \big) =  (-1)^{\sum_{\xi  : \langle \xi, \gamma \rangle > 0 } \langle \xi, \gamma \rangle \dim \bNtwo(\xi)} \M^{\bG, \bNone \oplus \bNtwo^\ast}_\gamma(f)
\end{equation}

\subsection{Coulomb branches of quiver gauge theories and Kac-Moody affine Grassmannian slices}
\label{sec:qgt}

Fix a quiver $(I,E)$ with no edge loops, which we think of as an orientation of a Dynkin diagram of symmetric Kac-Moody type. Note that multiple edges are allowed, for example as in the Dynkin diagram of the affine Lie algebra $\widehat{\mathfrak{sl}}_2$.  We keep the same notation from \S \ref{sec:generalized-affine-grassmannian-slices} for coweights, dominant coweights, simple roots, etc., in this general Kac-Moody setting.

Fix $\lambda \in \coweightlattice_{++}$ and $\mu \in \coweightlattice$ with $\lambda \geq \mu$. Define $\bv = (\sv_i)_{i \in I}$ and $\bw = (\sw_i)_{i \in I}$ by the same formulas \eqref{eq:d:104} and \eqref{eq:d:104b}. Fix finite dimensional $I$--graded $\CC$--vector spaces $V = \bigoplus_{i\in I} V_i$ and $W = \bigoplus_{i\in I}$, having dimensions $\dim V_i = \sv_i$ and $\dim W_i = \sw_i$.  Define:
\begin{align}
\bG & = \prod_{i \in I} \GL(V_i),   \label{eq:bG-and-bN-for-quiver-gauge-theory}\\
\bN & = \bigoplus_{a \in E} \Hom(V_{\source(a)}, V_{\target(a)}) \oplus \bigoplus_{i\in I} \Hom(W_i, V_i)   \label{eq:bG-and-bN-for-quiver-gauge-theoryb}
\end{align}
Finally, we pick bases of each $V_i$, which determines a pinning of $\bG$. In particular, we get a maximal torus $\bT$ of $\bG$.

In physics terminology, $\bG$ and $\bN$ determine a ``quiver gauge theory'', and associated to this data, we can form a Coulomb branch. Braverman, Finkelberg, and Nakajima prove the following astonishing theorem, which is the reason Coulomb branches enter our story.  We mention that this result can be interpreted physically in terms of moduli spaces of singular monopoles \cite{BFN2,BDG}. We state a version twisted by the diagram automorphism $\ast$ from Remark \ref{rmk: diagram aut}.
\begin{Theorem}[\mbox{\cite[Theorem 3.10]{BFN2}}]
  \label{thm:bfn-slices-are-Coulomb-branches}
 Suppose the quiver $(I,E)$ is finite ADE type. Let $\lambda$ and $\mu$ be as above, and let $\bG$ and $\bN$ be defined by \eqref{eq:bG-and-bN-for-quiver-gauge-theory}.  Then there is an explicit isomorphism:
  \begin{equation}
    \label{eq:d:55}
    \cM_{C}(\bG,\bN) \overset{\sim}{\rightarrow} \overline{\cW}^\lambda_\mu
  \end{equation}
\end{Theorem}

Motivated by this theorem, we have the following definition.
\begin{Definition}
  \label{def:kac-moody-slices}
Let $\lambda$ and $\mu$ be as above, and let $\bG$ and $\bN$ be defined by \eqref{eq:bG-and-bN-for-quiver-gauge-theory}. Then the \emph{(Kac-Moody) affine Grassmannian slice} is defined by:
  \begin{equation}
    \label{eq:d:52}
    \overline{\cW}^\lambda_\mu = \cM_{C}(\bG,\bN)
  \end{equation}
\end{Definition}
To match the notation from the previous section, we also write:
\begin{equation}
  \label{eq:d:53}
  \cA(\lambda,\mu) = \cA(\bG,\bN) = \kk[\overline{\cW}^\lambda_\mu] 
\end{equation}
When the quiver $(I,E)$ is affine type, $\overline{\cW}^\lambda_\mu$ is often called a \emph{double affine Grassmannian slice} (short for ``affine affine Grassmannian slice''). Finally, observe that $\overline{\cW}^\lambda_\mu$ does not depend on the orientation of the quiver because of the Fourier transform from \S\ref{subsec: FT}.

\subsubsection{Zastava spaces}
\label{section: zastava spaces via Coulomb}

In the special case where $\bG$ is a product of general linear groups, one can define a \emph{positive part} of the ind-scheme $\cR_{\bG,\bN}$, which we denote $\cR_{\bG,\bN}^{+}$. The Borel-Moore homology
\begin{equation}
  \label{eq:d:56}
\cA^{+}(\bG,\bN) = H^{\bG(\cO)}_\bullet(\cR_{\bG,\bN}^{+}) 
\end{equation}
forms a subalgebra of $\cA(\bG,\bN)$ and Braverman, Finkelberg, and Nakajima prove the following theorem.
\begin{Theorem}[\mbox{\cite[Corollary 3.4]{BFN2}}]
  \label{thm:bfn-zastava}
 Suppose the quiver $(I,E)$ is finite  ADE type. Let $\lambda$ and $\mu$ be as above, and let $\bG$ and $\bN$ be defined by \eqref{eq:bG-and-bN-for-quiver-gauge-theory}.  Then there is an explicit isomorphism:
  \begin{equation}
    \label{eq:d:55}
    \Spec \cA^{+}(\bG,\bN) \overset{\sim}{\rightarrow} \zastavap^{\lambda - \mu}
  \end{equation}
\end{Theorem}
Analogous to Definition \ref{def:kac-moody-slices}, in Kac-Moody type, we define the zastava space $\zastavap^{\lambda - \mu}$ by formula \eqref{eq:d:55}. To match the notation from the previous section, we also write:
\begin{equation}
  \label{eq:d:54}
  \cA^{+}(\lambda,\mu) = \cA^{+}(\bG,\bN) = \kk[\zastavap^{\lambda - \mu}] 
\end{equation}

Similarly, there is a \emph{negative part} $\cR_{\bG,\bN}^{-}$, we analogously define $\cA^{-}(\lambda,\mu) = \cA^{-}(\bG,\bN)$, and Theorem \ref{thm:bfn-zastava} holds analogously for $\zastavam^{\lambda - \mu}$.

\subsubsection{GKLO embedding}
\label{sec:beyond-finite-gklo-embedding}
At this point we can explain where Theorem \ref{thm:GKLO-embedding-of-slices} comes from. The ring $\tcA(\lambda-\mu)$ is identified with $\cA(\bT,0)$: the $u_{i,r}$ are particular generators for this torus Coulomb branch, and the $w_{i,r}$ are the equivariant parameters generating $H^\bullet_{\bT(\cO)}(\pt)$. The GKLO embedding is exactly the composed map:
\begin{equation}
  \label{eq:d:105}
  \cA(\bG,\bN) \overset{\iota_\ast^{-1}}{\hooklongrightarrow}  \cA(\bT,\bN)\left[  \beta^{-1}\right]_{\beta \in \Delta_{\bg}^+} \overset{\bz^*}{\hooklongrightarrow} \cA(\bT,0)\left[  \beta^{-1}\right]_{\beta \in \Delta_{\bg}^+}
\end{equation}
Note that the positive roots  $\beta = w_{i,r} - w_{i,s}$ for $i\in I$ and $1 \leq r < s \leq \sv_i$. Thus Theorem \ref{thm:GKLO-embedding-of-slices} holds exactly the same in the Kac-Moody case. The only difference is that the regular functions $Q_i(z)$ and $P_i(z)$ are not defined via matrix coefficients, but rather are characterized uniquely by formulas \eqref{eq:d:13} and \eqref{eq:d:15}. Similarly, the GKLO embedding \eqref{eq:d:7} for $\zastavap^{\lambda - \mu}$ also holds exactly as before.

\subsubsection{Fundamental monopole operators}
\label{sec: FMOs}

The definition of FMOs (Definition \ref{def: FMOs}), the fact that they are regular functions (Theorem \ref{thm: FMOs are regular}), and the fact that they generate the coordinate ring (Theorem \ref{thm: FMOs generate for slice}) continue to hold without change. We can now say where FMOs come from: they are particular examples of minuscule monopole operators.

Let $\bm = (\sm_i)_{i \in I} \in \ZZ^I$ with $\mathbf{0} \leq \bm \leq \bv$. We define $\varpi_\bm$ to be the following coweight of $\bG$: it is the sum over $i$ of the $\sm_i$-th fundamental coweight $(1,\ldots,1,0,\ldots, 0)$ of the factor $\GL(V_i)$ of $\bG$. For this coweight, we have an identification between $\Lambda^\bv_{\bm}$ and $H_\bullet^{\Stab_{t^{\gamma}}}(\pt)$.
The coweight $\varpi_\bm$ is minuscule, and for $f \in \Lambda^\bv_{\bm}$, and similarly to (\cite[Proposition A.2]{BFN2}) we have:
\begin{equation}
  \label{eq:d:107}
  \M_\bm^{+}(f) = \M^{\bG,\bN}_{\varpi_\bm}(f)
\end{equation}
The coweight $-\varpi_\bm$ is also minuscule, and we have an identification between $\Lambda^\bv_{\bm}$ and $H_\bullet^{\Stab_{t^{\gamma}}}(\pt)$. For $f \in \Lambda^\bv_{\bm}$, we have
\begin{equation}
  \label{eq:d:108}
  \M_\bm^{-}(f) = (-1)^{\text{sign}} \cdot \M^{\bG,\bN}_{-\varpi_\bm}(f)
\end{equation}
where ``$\text{sign}$'' is given by \eqref{eq:d:48}.

\begin{Remark}
Note that the coweight $-\varpi_\bm$ is antidominant.  This is convenient notationally.  We could instead consider its dominant Weyl translate $\varpi_\bm^\ast = - w_0 \varpi_\bm$, as in \cite[\S A(ii)]{BFN2}.
\end{Remark}

\section{Embeddings of Kac-Moody affine Grassmannian slices}

Our main goal in this section is to generalize Theorem \ref{thm:compatibility-of-FMOs-with-embeddings-of-slices} to Kac-Moody affine Grassmannian slices. The main difficulty is that we do not \emph{a priori} have a definition of the closed embedding \eqref{eq:d:10}, because we only have the Coulomb branch definition of Kac-Moody affine Grassmannian slices. Because we know generation by FMOs, we could try to use Theorem \ref{thm:compatibility-of-FMOs-with-embeddings-of-slices} to define the closed embedding. Unfortunately, this does not work because we do not know \emph{a priori} that the map on generators extends to an algebra homomorphism. 

Instead, we will proceed another way and construct the closed embedding \eqref{eq:d:10} using the Coulomb branch definition. Theorem \ref{thm:compatibility-of-FMOs-with-embeddings-of-slices} appears during the proof of the construction.

\subsection{Conicity condition and closed embeddings}

Let $\lambda' \in \coweightlattice_{++}$ with $\lambda \geq \lambda' \geq \mu$. In finite type, we automatically have a closed embedding of $\overline{\cW}^{\lambda'}_{\mu}$ into $\overline{\cW}^{\lambda}_{\mu}$, but in general Kac-Moody type, we need a further condition on $\lambda'$ to obtain a closed embedding.

Let $C$ denote the Cartan matrix of our Kac-Moody type, which is a symmetric matrix by assumption.

\begin{Definition}
  \label{def:conical-coweights}
  Let $\lambda, \lambda' \in \coweightlattice_{++}$ with $\lambda \geq \lambda'$. Define $\bw$ as above and $\bv'' = (\sv''_i)_{i \in I}$ by:
  \begin{equation}
    \label{eq:d:62}
 \sv''_i = \langle \lambda - \lambda', \fundweight_i \rangle   
  \end{equation}
  We say that $\lambda'$ satisfies the \emph{conicity condition} for $\lambda$ if 
  \begin{equation}
    \label{eq:d:63}
     \bu \cdot (\bw-C\bv'')  + \bu \cdot ( C \bu ) \geq 1 
  \end{equation}
  for all non-zero  $\bu$  with  $\mathbf{0} \leq \bu \leq \bv''$. Here $\cdot$ denotes the usual dot product on $\ZZ^{I}$. 
\end{Definition}

\begin{Remark}
  As we will recall in Appendix \ref{sec:appendix}, the conicity condition corresponds exactly to quiver gauge theories that are either ``good'' or ``ugly'' in the physics terminology.
\end{Remark}

\begin{Remark}
The conicity condition is automatic in finite type, and for most affine types.  See Theorem \ref{thm: conicity finite affine} for details.
\end{Remark}

\begin{Proposition}
  If $\lambda'$ satisfies the conicity condition for $\lambda$, then the Kac-Moody slice $\overline{\cW}^\lambda_{\lambda'}$ is a conical variety. That is,  $\cA(\lambda, \lambda') = \kk[\overline{\cW}^\lambda_{\lambda'}]$ admits a grading by $\ZZ_{\geq 0}$ with the $0$-graded piece consisting just of $\kk$.
\end{Proposition}

This result is a special case of Corollary \ref{appendix:main result}. In particular, we obtain an algebra homomorphism:
\begin{equation}
  \label{eq:d:70}
\cA(\lambda, \lambda') \twoheadrightarrow \kk
\end{equation}
which sends all elements of non-zero degree to zero.  More precisely, let $\bm = (\sm_i)_{i\in I} \in \ZZ^I$ with $0 \leq \sm_i \leq \sv_i''$ for all $i \in I$, and let $f \in \Lambda^{\bv''}_\bm$. Then the map \eqref{eq:d:70} is defined by:
\begin{equation}
  \label{eq:d:106}
   \M_{\bm}^{\pm}(f) \mapsto
   \begin{cases}
    0, & \text{if } \bm \neq \mathbf{0} \\ f(0), & \text{if }  \bm = \mathbf{0}  
   \end{cases}
\end{equation}
The above notation $f(0)$ means setting all equivariant parameters in $f$ to zero.

We can now state the main theorem of this section.
\begin{Theorem}
  \label{thm:closed-embedding-of-KM-slices}
  Let $\lambda,\lambda' \in \coweightlattice_{++}$ and $\mu \in \coweightlattice$ with $\lambda \geq \lambda' \geq \mu$. Further assume that $\lambda'$ satisfies the conicity condition for $\lambda$. Then there is a closed embedding
  \begin{equation}
    \label{eq:d:64}
 \overline{\cW}^{\lambda'}_\mu \hookrightarrow  \overline{\cW}^\lambda_\mu
  \end{equation}
under which the FMOs restrict exactly as in Theorem \ref{thm:compatibility-of-FMOs-with-embeddings-of-slices}. 
\end{Theorem}

\subsection{Proving Theorem \ref{thm:closed-embedding-of-KM-slices}}

Let $\lambda$, $\lambda'$, and $\mu$ be as in the Theorem statement. Let $\bw$ and $\bv$ be as above, $\bv''$ as in Definition \ref{def:conical-coweights}, and $\bv'$ as in \S \ref{sec:Compatibility of FMOs with inclusions of generalized affine Grassmannian slices}.

\subsubsection{Levi restriction}

We need the following Levi restriction statement, generalizing the localization map from \S \ref{sec: localization}. Since we are in the setting of quiver gauge theories $\cA(\bG,\bN)$ is generated by FMOs, and in particular by the minuscule monopole operators $\M_\gamma^{\bG,\bN}(f)$.

\begin{Proposition}
  \label{prop:levi-restriction}
Let $\bL$ be a standard Levi of $\bG$, and let $\Delta_{\bL}^+ \subset \Delta_{\bG}^+$ denote set of positive roots for $\bL$. There is an embedding
\begin{equation}
  \label{eq:d:57}
\iota_\ast^{-1} : \cA(\bG,\bN) \ \hookrightarrow \ \cA(\bL, \bN) \left[\beta^{-1}\right]_{\beta \in \Delta_{\bG}^+ \setminus \Delta_\bL^+}
\end{equation}
uniquely determined by the commutativity of the following diagram: 
\begin{equation*}
  \label{eq:d:58}
\begin{tikzcd}
\cA(\bG,\bN) \ar[rr,hook,"\iota_\ast^{-1}"] \ar[rd,hook,"\iota_\ast^{-1}"'] & & \cA(\bL, \bN) \left[\beta^{-1}\right]_{\beta \in \Delta_{\bG}^+ \setminus \Delta_\bL^+} \ar[ld,hook',"\iota_\ast^{-1}"] \\
& \cA(\bT,\bN)\left[  \beta^{-1}\right]_{\beta \in \Delta_{\bG}^+} 
\end{tikzcd}
\end{equation*}
\end{Proposition}
\begin{proof}
Because the MMOs generate $\cA(\bG,\bN)$, it suffices to show that their images in \linebreak $\cA(\bT,\bN)\left[  \beta^{-1}\right]_{\beta \in \Delta_{\bG}^+} $ are equal to the  images of elements of $\cA(\bL, \bN) \left[\beta^{-1}\right]_{\beta \in \Delta_{\bG}^+ \setminus \Delta_\bL^+}$. In fact, we will show that the image of each MMO is a linear combination of MMOs from $\cA(\bL, \bN)$ with coefficients only involving the allowed denominators. 

Let $\gamma$ be a minuscule dominant weight for $\bG$, and let $f \in H^\bullet_{\bT}(\pt)^{W_{\bG,\gamma}}$ be a dressing for $\gamma$. Then under localization to the torus, the corresponding MMO is:
  \begin{equation}
    \label{eq:d:82}
\iota_\ast^{-1} \big(M_\gamma^{\bG,\bN}(f) \big) =  \sum_{\sigma \in W_{\bG}/ W_{\bG,\gamma} } \sigma(f)  \cdot \frac{r_{\sigma(\gamma)}}{ \sigma \left( \prod_{\beta \in \Delta_{\bG}^{+} : \langle \beta, \gamma \rangle > 0} \beta \right) }
  \end{equation}

  Let $\{ \gamma_1, \ldots, \gamma_N \}$ be the elements of $W_{\bG} \cdot \gamma$ that are dominant for $\bL$.  For each $1 \leq j\leq N$, choose $\sigma_j \in W_\bG$ such that $\gamma_j = \sigma_j(\gamma)$. Then \eqref{eq:d:82} is equal to:
  \begin{equation}
    \label{eq:d:83}
   \sum_{j = 1}^{N} \sum_{\sigma \in W_{\bL}/ W_{\bL,\gamma_j} } \sigma\sigma_j (f)  \cdot \frac{r_{\sigma(\gamma_j)}}{ \sigma \left( \prod_{\beta \in \Delta_{\bL}^{+}  : \langle \beta, \gamma_j \rangle > 0} \beta  \cdot \prod_{\eta \in \Delta_{\bG}^{+} \backslash \Delta_{\bL}^+  : \langle \eta, \gamma_j \rangle > 0} \eta \right) }
  \end{equation}
Consider a fixed summand indexed by $j \in [N]$. Multiplying that summand by
\begin{equation}
  \label{eq:d:84}
  \prod_{{\tau \in W_{\bL}/ W_{\bL,\gamma_j} }} \tau \left( \prod_{\eta \in \Delta_{\bG}^{+} \backslash \Delta_{\bL}^+  : \langle \eta, \gamma_j \rangle > 0} \eta \right)
\end{equation}
on denominator and numerator, we get the image under $\iota_\ast^{-1}$ of
\begin{equation}
  \label{eq:d:86}
  \frac{1}{\prod_{{\tau \in W_{\bL}/ W_{\bL,\gamma_j} }} \tau \left( \prod_{\eta \in \Delta_{\bG}^{+} \backslash \Delta_{\bL}^+  : \langle \eta, \gamma_j \rangle > 0} \eta \right)} \cdot M^{\bL,\bN}_{\gamma_i}\left(   
\sigma_j(  f) \cdot
  \prod_{{\tau \in W_{\bL}/ W_{\bL,\gamma_j} } : \tau \neq e} \tau \left( \prod_{\eta \in \Delta_{\bG}^{+} \backslash \Delta_{\bL}^+  : \langle \eta, \gamma_j \rangle > 0} \eta \right) \right)
\end{equation}
which is an element of $\cA(\bL, \bN) \left[\beta^{-1}\right]_{\beta \in \Delta_{\bG}^+ \setminus \Delta_\bL^+}$.
\end{proof}

In formula \eqref{eq:d:36} for MMOs, we can abuse notation and formally consider the dressing to be an element of the fraction field of $H^\bullet_{\Stab_{t^{\gamma}}}(\pt)$. With this generalization, we can write \eqref{eq:d:86} more compactly as
\begin{equation}
  \label{eq:d:87}
\iota_\ast^{-1} \big( \M_\gamma^{\bG,\bN}(f) \big) = \sum_j M^{\bL,\bN}_{\gamma_j} \left( \frac{\sigma_j(f)}{\prod_{\substack{\beta \in \Delta_\bG \setminus \Delta_\bL, \\ \langle \beta, \gamma_j\rangle < 0}}\beta}\right)
\end{equation}
where $\sigma_j \in W_\bG$ is any element such that $\gamma_j = \sigma_j(\gamma)$.

\subsubsection{Constructing the closed embedding}
\newcommand{\bNprimew}{{\bN'_\bw}}
\newcommand{\bNprimeprimew}{{\bN''_\bw}}
\newcommand{\bNmix}{{\bN^{\mathrm{mix}}}}
\newcommand{\bNmixone}{{\bN_1^{\mathrm{mix}}}}
\newcommand{\bNmixtwo}{{\bN_2^{\mathrm{mix}}}}

For each $i \in I$, we decompose 
\begin{equation}
V_i = V_i'\oplus V_i''
\end{equation}
where $V_i'$ is the span of the first $\sv_i'$ basis vectors of $V_i$, and $V_i''$ is the span of the remaining $\sv_i''$ basis vectors. Taking the product over $i\in I$ of the corresponding subgroups $\GL(V_i') \times \GL(V_i'') \subset \GL(V_i)$, we obtain a standard Levi subgroup $\bG' \times \bG'' \subset \bG$.  As a representation of $\bG' \times \bG''$, we have
\begin{equation}
\bN  \ = \ \bNprimew \oplus \bNmix \oplus \bNprimeprimew
\end{equation}
where: 
\begin{align}
\bNprimew & = \bigoplus_{a \in E} \Hom(V'_{\source(a)}, V'_{\target(a)}) \oplus \bigoplus_{i\in I} \Hom(W_i, V'_i), \\
\bNmix & = \bigoplus_{a \in E} \Hom(V'_{\source(a)}, V''_{\target(a)}) \oplus \bigoplus_{b \in E} \Hom(V''_{\source(b)}, V'_{\target(b)}), \\
\bNprimeprimew &= \bigoplus_{a \in E} \Hom(V''_{\source(a)}, V''_{\target(a)}) \oplus \bigoplus_{i \in I} \Hom(W_i, V''_i)
\end{align}

By definition, we have
\begin{equation}
  \label{eq:d:69}
  \overline{\cW}_{\lambda'}^\lambda = \cM_C(\bG'', \bNprimeprimew)
\end{equation}
We have assumed that $\lambda'$ satisfies the conicity condition for $\lambda$, so there is an algebra homomorphism as in \eqref{eq:d:70}:
\begin{equation}
  \label{eq:d:71}
 \cA(\bG'', \bNprimeprimew)  \twoheadrightarrow \kk
\end{equation}
By the K\"unneth formula we also have \cite[\S 3(vii)(a)]{BFN1}:
\begin{equation}
  \label{eq:d:72}
  \cA(\bG' \times \bG'', \bNprimew \oplus \bNprimeprimew) = \cA(\bG',\bNprimew) \otimes \cA(\bG'',\bNprimeprimew) 
\end{equation}
Combining this with \eqref{eq:d:71} gives us a surjective map:
\begin{equation}
  \label{eq:d:73}
  \cA(\bG' \times \bG'', \bNprimew \oplus \bNprimeprimew) \twoheadrightarrow  \cA(\bG',\bNprimew) 
\end{equation}

\begin{Proposition}
  \label{prop:killing-cone}
There is a surjective map $\pi: \cA(\bG'\times \bG'', \bN) \twoheadrightarrow \cA(\bG', \bNprimew \oplus \bNmix ) $ which is uniquely determined by the commutativity of the following square:
\begin{equation}
  \label{eq:d:74}
\begin{tikzcd}
\cA(\bG'\times \bG'', \bN) \ar[d,hook,"\bz^\ast"] \ar[r,two heads,"\pi"] & \cA(\bG', \bNprimew \oplus \bNmix ) \ar[d,hook,"\bz^\ast"] \\
\cA(\bG' \times \bG'', \bNprimew \oplus \bNprimeprimew)  \ar[r,two heads] & \cA(\bG',\bNprimew)
\end{tikzcd}
\end{equation}
where the vertical arrows are obtained by forgetting $\bNmix$, and the lower horizontal arrow is \eqref{eq:d:73}.  

Furthermore, let $\gamma = (\gamma',\gamma'')$ be a minuscule coweight for $\bG' \times \bG''$, and let $f^{(1)} \otimes f^{(2)} \in H^\bullet_{\Stab_{t^{\gamma'}}}(\pt) \otimes H^\bullet_{\Stab_{t^{\gamma''}}}(\pt) = H^\bullet_{\Stab_{t^{\gamma}}}(\pt)$. Then we have
\begin{equation}
\label{eq:d:88}
 \pi \left( M^{\bG' \times \bG'', \bN}_{(\gamma',\gamma'')}\big(f^{(1)} \otimes f^{(2)}\big) \right) =  \begin{cases}   M^{\bG', \bNprimew \oplus \bNmix}_{\gamma'}\left(f^{(1)} \otimes f^{(2)}(0)\right), & \ \text{if } \gamma'' = 0 \\ 0, & \ \text{if } \gamma'' \neq 0 \end{cases}
\end{equation}

where $f^{(2)} \mapsto f^{(2)}(0)$ is the map $H^\bullet_{\Stab_{t^{\gamma''}}}(\pt) \rightarrow H^0_{\Stab_{t^{\gamma''}}}(\pt) = \kk$.
\end{Proposition}
\begin{proof}

 Clearly $\pi$ is a well-defined rational map. For the rational map $\pi$, the formulas \eqref{eq:d:88}  follow by direct calculation using \eqref{eq:d:67} and \eqref{eq:d:106}. 

 By \cite[Remark 3.8]{Weekes1}, $\cA(\bG' \times \bG'', \bNprimew \oplus \bNprimeprimew)$ and $\cA(\bG',\bNprimew)$ are each generated by MMOs. Formula \eqref{eq:d:88} tells us that all the MMOs map to MMOs, and that each MMO is the image of an MMO. So we conclude that $\pi$ is a regular map, and it is surjective.
\end{proof}

\newcommand{\bNthree}{{\bN^{(3)}}}
\newcommand{\bNfour}{{\bN^{(4)}}}
\newcommand{\bNfourp}{{\bN^{(4)}_+}}
\newcommand{\bNfourm}{{\bN^{(4)}_-}}

We can write $\bNmix = \bNmixone  \oplus \bNmixtwo$ where:
\begin{align}
  \label{eq:d:75}
  \bNmixone = \bigoplus_{a \in E} \Hom(V'_{\source(a)}, V''_{\target(a)}) \\
  \bNmixtwo = \bigoplus_{b \in E} \Hom(V''_{\source(b)}, V'_{\target(b)})
\end{align}
The Fourier transform from \S \ref{subsec: FT} provides us with a isomorphism
\begin{equation}
  \label{eq:d:76}
  \cF : \cA(\bG', \bNprimew \oplus \bNmix) \stackrel{\sim}{\longrightarrow} \cA\big(\bG', \bNprimew \oplus (\bNmixone)^* \oplus \bNmixtwo\big)
\end{equation}
For each $i \in I$, we fix a vector subspace
\begin{equation}
  \label{eq:d:77}
  W_i' \ \subseteq \ W_i \oplus \bigoplus_{a \in E : \source(a) = i} V_{\target(a)}'' \oplus \bigoplus_{b \in E : \target(a) = i} V_{\source(a)}'' 
\end{equation}
with $\dim W_i' = \sw_i'$. This is possible because one computes the right hand side of \eqref{eq:d:77} to have dimension $\sw_i' + 2 \sv_i''$. Let $X_i$ be any complementary vector subspace. Then we define
\begin{equation}
  \label{eq:d:78}
  \bNthree = \bigoplus_{a \in E} \Hom(V'_{\source(a)}, V'_{\target(a)}) \oplus \bigoplus_i \Hom(W'_i, V'_i)
\end{equation}
and:
\begin{equation}
  \label{eq:d:79}
  \bNfour = \bigoplus_i \Hom(X_i,V_i')
\end{equation}
In particular, we have:
\begin{equation}
\label{eq: the reveal}
\cM_C(\bG',\bNthree) \ = \ \overline{\cW}_\mu^{\lambda'}
\end{equation}
By construction, we can $\bG'$-equivariantly identify:
\begin{equation}
  \label{eq:d:80}
 \bNprimew \oplus (\bNmixone)^* \oplus \bNmixtwo) = \bNthree \oplus \bNfour 
\end{equation}
So we view \eqref{eq:d:76} as an isomorphism:
\begin{equation}
  \label{eq:d:81}
  \cF : \cA(\bG', \bNprimew \oplus \bNmix) \stackrel{\sim}{\longrightarrow} \cA(\bG', \bNthree \oplus \bNfour)
\end{equation}

For each $i \in I$, let us further choose a vector-space decomposition $X_i = X_i^{+} \oplus X_i^{-}$ with $\dim X_i^{+} = \dim X_i^{-} = \sv_i''$, and write 
\begin{equation}
  \label{eq:d:90}
  \bNfourp  = \bigoplus_i \Hom(X^+_i,V_i')
\end{equation}
and 
\begin{equation}
  \label{eq:d:91}
  \bNfourm  = \bigoplus_i \Hom(X^-_i,V_i')
\end{equation}
Then we further have a Fourier transform isomorphism
\begin{equation}
  \label{eq:d:92}
  \cF : \cA(\bG', \bNthree \oplus \bNfour)  \stackrel{\sim}{\longrightarrow}  \cA\big(\bG', \bNthree \oplus \bNfourp \oplus (\bNfourm)^* \big)
\end{equation}

To prove Theorem \ref{thm:closed-embedding-of-KM-slices}, we prove the following more elaborate statement. 
\begin{Theorem}[Elaborated version of Theorem \ref{thm:closed-embedding-of-KM-slices}]
  \label{thm:big-tikz-diagram-defining-closed-embedding}
There is a surjective homomorphism  $\cA(\lambda, \mu) = \cA(\bG,\bN) \twoheadrightarrow  \cA(\bG', \bN^{(3)}) = \cA(\lambda',\mu)$ that is uniquely determined by the commutativity of the following diagram:
\begin{equation}
\label{eq: big tikz diagram}
\begin{tikzcd}\cA(\bG,\bN) \ar[r, hook, "\iota_\ast^{-1}"] \ar[dddd,dashed,two heads] & \cA(\bG'\times \bG'', \bN)\Big[ \prod_{i\in I} \prod_{\substack{1\leq r \leq \bv'_i, \\ \bv'_i < s \leq \bv_i}}(w_{i,r} - w_{i,s})^{-1} \Big] \ar[d,two heads,"\pi"] \\
& \cA(\bG', \bNprimew \oplus \bNmix)\Big[\prod_{i\in I}\prod_{1\leq r \leq \bv_i'} w_{i,r}^{-1} \Big] \ar[d,"\cF"]\\
& \cA(\bG',\bNthree\oplus\bNfour)\Big[\prod_{i\in I}\prod_{1\leq r \leq \bv_i'} w_{i,r}^{-1} \Big]\ar[d,"\cF"]\\
& \cA\big(\bG', \bNthree \oplus \bNfourp \oplus (\bNfourm)^* \big)\Big[\prod_{i\in I}\prod_{1\leq r \leq \bv_i'} w_{i,r}^{-1} \Big]\ar[d,hook,"\bz^\ast"]\\
\cA(\bG', \bNthree) \ar[r,hook] &\cA(\bG',\bNthree)\Big[\prod_{i\in I}\prod_{1\leq r \leq \bv_i'} w_{i,r}^{-1} \Big]
\end{tikzcd}
\end{equation}
Additionally, under this map, the FMOs restrict exactly as in Theorem \ref{thm:compatibility-of-FMOs-with-embeddings-of-slices}.
\end{Theorem}
\begin{proof}

  Let $\bm = (\sm_i)_{i \in I} \in \ZZ^I$ with $\mathbf{0} \leq \bm \leq \bv$, and let $f \in \Lambda^\bv_\bm$. Recall from \eqref{eq:d:107} that we have:
  \begin{equation}
    \label{eq:d:93}
    \M_\bm^{+}(f) = \M_{\varpi_\bm}^{\bG,\bN}(f) 
  \end{equation}
  By Propositions \ref{prop:levi-restriction} and \ref{prop:killing-cone}, we have:
  \begin{equation}
    \label{eq:d:94}
   \pi \circ \iota_*^{-1} \left( \M_{\varpi_\bm}^{\bG,\bN}(f) \right)  = \M_{\varpi_\bm}^{\bG',\bNprimew \oplus \bNmix} \left(\frac{\tilde{f}}{ \prod_{i\in I} \prod_{p = 1}^{m_i} w_{i,p}^{\sv_i''} } \right) 
  \end{equation}
  if $\bm \leq \bv'$. If not, we get zero.

  All the weights of $\bNmixone$ pair negatively with $\varpi_{\bm}$, so the first Fourier transform acts by the identity, see \eqref{eq: FT for MMOs}. However, the second Fourier transform introduces a sign, sending \eqref{eq:d:94} to:
  \begin{equation}
    \label{eq:d:95}
    (-1)^{ \sum_{i \in I} \sm_i \sv_i''}\M_{\varpi_\bm}^{\bG',\bNthree \oplus \bNfourp \oplus (\bNfourm)^* }\left(\frac{\tilde{f}}{ \prod_{i\in I} \prod_{p = 1}^{m_i} w_{i,p}^{\sv_i''} } \right) 
  \end{equation}
Finally, $z^*$ sends this to:
\begin{equation}
  \label{eq:d:96}
    (-1)^{ \sum_{i \in I} \sm_i \sv_i''} \cdot \M_{\varpi_\bm}^{\bG',\bNthree}\left(\prod_{i\in I} \prod_{r=1}^{\sm_i} (-w_{i,r})^{\sv_i''} \cdot \frac{\tilde{f}}{ \prod_{i\in I} \prod_{p = 1}^{m_i} w_{i,p}^{\sv_i''} } \right) 
\end{equation}
Observe that the signs cancels, and this is exactly equal  to
\begin{equation}
  \label{eq:d:97}
    \M_\bm^{+}(\tilde{f}) = \M_{\varpi_\bm}^{\bG',\bNthree}(\tilde{f}) 
\end{equation}

Now we turn to the negative FMOs. Analogously to before we have:
  \begin{equation}
      \label{eq:d:98}
   \pi \circ \iota_*^{-1} \left( \M_{-\varpi_\bm}^{\bG,\bN}(f) \right)  = \M_{\varpi_\bm}^{\bG',\bNprimew \oplus \bNmix}\left(\frac{\tilde{f}}{ \prod_{i\in I} \prod_{p = 1}^{m_i} (-w_{i,p})^{\sv_i''} } \right) 
  \end{equation}
  if $\sm_i \leq \sv_i'$ for all $i \in I$, and zero otherwise.
This time the first Fourier transform introduces a sign sending \eqref{eq:d:98} to
  \begin{equation}
    \label{eq:d:99}
   (-1)^{\sum_{a \in E} \sm_{\source(a)} \sv_{\target(a)}''} \cdot  \M_{\varpi_\bm}^{\bG',\bNthree\oplus\bNfour}\left(\frac{\tilde{f}}{ \prod_{i\in I} \prod_{p = 1}^{m_i} (-w_{i,p})^{\sv_i''} } \right) 
  \end{equation}
and the second Fourier transform acts by the identity.
Finally, $z^*$ sends this to:
\begin{equation}
  \label{eq:d:100}
   (-1)^{\sum_{a \in E} \sm_{\source(a)} \sv_{\target(a)}''} \cdot  \M_{\varpi_\bm}^{\bG',\bNthree}\left( \prod_{i\in I} \prod_{r=1}^{\sm_i} (w_{i,r})^{\sv_i''} \cdot \frac{\tilde{f}}{ \prod_{i\in I} \prod_{p = 1}^{m_i} (-w_{i,p})^{\sv_i''} } \right) 
\end{equation}
The signs do not cancel, and we get:
\begin{equation}
  \label{eq:d:101}
 (-1)^{\sum_{i\in I} \sm_i \sv_i'' + \sum_{a \in E} \sm_{\source(a)} \sv_{\target(a)}''} \cdot \M_{-\varpi_\bm}^{\bG',\bNthree}(\tilde{f})  
\end{equation}

Recall \eqref{eq:negative-FMOs-images-under-GKLO} that the negative FMO $\M_{\bm}^{-}(f)$ is defined as a sign times $\M_{-\varpi_\bm}^{\bG',\bNthree}(f)$.  In particular, the sign appearing \eqref{eq:d:101} exactly compensates for this sign, and we conclude that  
\begin{equation}
  \label{eq:d:102}
\M_{\bm}^{-}(f) \mapsto \M_{\bm}^{-}(\tilde{f})
\end{equation}

As both source and target of the left vertical arrow in \eqref{eq: big tikz diagram} are generated by FMOs, we conclude that the map is regular and surjective. 

\end{proof}

\begin{Remark}
  We do not know \emph{a priori} that this closed embedding is the same as the closed embedding \eqref{eq:d:10} defined group-theoretically for finite-type affine Grassmannian slices. Rather, we only know this \emph{a posteriori} because both maps have the same effect on FMOs. In particular, Theorem \ref{thm:big-tikz-diagram-defining-closed-embedding} does not subsume Theorem \ref{thm:compatibility-of-FMOs-with-embeddings-of-slices}.
\end{Remark}

\subsection{Poisson structure}
\label{subsection: Poisson structure}

Coulomb branch coordinate rings each come equipped with a one-parameter non-commutative deformation quantization, see \cite[\S 3(iv)]{BFN1}. Geometrically, this extra parameter $\hbar$ corresponds to considering equivariance under an additional $\CC^\times$ corresponding to ``loop rotation'', and we obtain an algebra over $\kk[\hbar] = H_{\CC^\times}^\bullet(pt)$. This deformation quantization endows the commutative Coulomb branch with a Poisson structure. So we might ask when the closed embedding constructed in Theorem \ref{thm:closed-embedding-of-KM-slices} respects the Poisson structure. To answer this question, we need the following strengthening of the conicity condition. 

\begin{Definition}
  \label{def:good-coweights}
  Let $\lambda, \lambda' \in \coweightlattice_{++}$ with $\lambda \geq \lambda'$. Define $\bw$ as above and $\bv'' = (\sv''_i)_{i \in I}$ by:
  \begin{equation}
    \label{eq:d:62}
 \sv''_i = \langle \lambda - \lambda', \fundweight_i \rangle   
  \end{equation}
  We say that $\lambda'$ satisfies the \emph{good condition} for $\lambda$ if 
  \begin{equation}
    \label{eq:d:63}
     \bu \cdot (\bw-C\bv'')  + \bu \cdot ( C \bu ) \geq 2 
  \end{equation}
  for all non-zero  $\bu$  with  $\mathbf{0} \leq \bu \leq \bv''$. Here $\cdot$ denotes the usual dot product on $\ZZ^{I}$. 
\end{Definition}

\begin{Theorem}
  \label{thm:closed-poisson-embedding-of-KM-slices}
  Let $\lambda, \lambda' \in \coweightlattice_{++}$ and $\mu \in \coweightlattice$ with $\lambda \geq \lambda' \geq \mu$. Further assume that $\lambda'$ satisfies the good condition for $\lambda$. Then the closed embedding
  \begin{equation}
    \label{eq:d:64}
 \overline{\cW}^{\lambda'}_\mu \hookrightarrow  \overline{\cW}^\lambda_\mu
  \end{equation}
  is compatible with Poisson structures.
\end{Theorem}

\begin{Remark}
  \label{rem:poisson-subvariety-does-not-quantize}
One might ask if Theorem \ref{thm:closed-poisson-embedding-of-KM-slices} holds for the quantized Coulomb branches. Remarkably, the answer is no! The simplest example is the case of the $A_1$ quiver, $\lambda = \simplecoroot$, and $\lambda' = 0$. In this case, $\overline{\cW}^\lambda_\mu$ is the nilpotent cone for $\SL_2$, and $\overline{\cW}^{\lambda'}_\mu$ is a point. The map \eqref{eq:d:64} quantizes precisely if the Coulomb branch quantization of $\overline{\cW}^\lambda_\mu$ has a non-trivial one-dimensional module, and this property fails in this example.
  
More precisely, let us work over $\kk = \CC$. The Coulomb branch quantization $\cA_\hbar(\lambda,\mu)$ is an algebra  over $\CC[\hbar]$, and we seek a surjection $\cA_\hbar(\lambda,\mu) \twoheadrightarrow \cA_\hbar(\lambda',\mu) = \CC[\hbar]$ over $\CC[\hbar]$. Specializing at $\hbar =1$, we would obtain a non-trivial one-dimensional module for the algebra $\cA_{\hbar=1}(\lambda,\mu)$.  In our example, one can check that the Coulomb branch quantization $\cA_{\hbar=1}(\lambda,\mu)$ is a particular central quotient of $U(\fsl_2)$ (corresponding to $-\rho$), which does not have a one-dimensional module.   (Of course, there does exist a quantization possessing a one-dimensional module, but this corresponds to a different central quotient of $U(\fsl_2)$.) This is a fundamental reason why the construction of Theorem \ref{thm:closed-embedding-of-KM-slices} is necessarily complicated: any simple geometric construction would also make sense for the quantization. It is all the more remarkable that we obtain such a nice formula for the restriction of the geometrically defined FMOs.
\end{Remark}

\begin{proof}[Proof of Theorem \ref{thm:closed-poisson-embedding-of-KM-slices}]
It suffices to check that the top horizontal map and all the right vertical maps in \eqref{eq: big tikz diagram} are Poisson. All the maps except $\pi$ are Poisson by the results of \cite{BFN1}. As for $\pi$, using Corollary \ref{appendix:main result}  the good condition implies that $\pi$ is Poisson.
\end{proof}

\subsection{Adding defect map and closed embeddings for zastava spaces}

We do not have an \emph{a priori} definition of adding defect like Definition \ref{def:adding-defect}. Instead, we first define the map \eqref{eq:d:22} on localized rings by the exact same formulas. Now we claim that there is a (necessarily uniquely determined) algebra homomorphism $\kk[\zastavap^{\gamma}] \rightarrow \kk[ \zastavap^{\gamma'} \times \AA^{(\gamma'')}]$ making the square \eqref{eq:d:24} commute. We prove this by checking that the FMOs, which generate $\kk[\zastavap^{\gamma}]$, map to elements of $\kk[ \zastavap^{\gamma'} \times \AA^{(\gamma'')}]$ under the map \eqref{eq:d:22} of localized rings: this is exactly what we computed in the proof of Theorem \ref{thm:compatibility-of-FMOs-with-adding-defect-map-of-zastava}.

As in the discussion after Theorem \ref{thm:compatibility-of-FMOs-with-adding-defect-map-of-zastava}, we can restrict the adding defect map along the unique colored divisor supported at $0$. This gives us a map of zastava spaces satisfying Theorem \ref{thm:compatibility-of-FMOs-with-embeddings-of-zastava}. Therefore the map on coordinate rings is surjective, and the map of zastava spaces is a closed embedding. We see that it is much easier to construct the closed embedding for zastava spaces than for Kac-Moody affine Grassmannian slices.


%% file: appendix.tex
\appendix
\section{Good and ugly conditions}
\label{sec:appendix}

Cremonesi, Hanany and Zaffaroni \cite{CHZ} proposed a certain formal infinite sum called the \emph{monopole formula}, which encodes a $\ZZ$--grading on the coordinate ring of a Coulomb branch.  We will  briefly recall the relevant details, before turning to the case of quiver gauge theories.

As in \S\ref{subsection: Coulomb branches in general}, we consider the Coulomb branch $\cM_C(\bG,\bN)$ associated to a pair $(\bG,\bN)$ of a reductive group $\bG$ and its representation $\bN$.  For any coweight $\gamma$ of $\bG$, define:
\begin{equation}
\label{appendix 1}
\Delta(\gamma) \ = \ -\sum_{\alpha} |\langle \alpha, \gamma\rangle | + \tfrac{1}{2} \sum_\xi |\langle \xi, \gamma\rangle| \dim \bN(\xi)
\end{equation}
The first sum runs over the positive roots $\alpha$ of $\bG$, while the second  runs over the weights $\xi$ of $\bN$.

Recall that following Braverman, Finkelberg and Nakajima, we define $\cM_C(\bG,\bN)$ as the spectrum of the ring $\cA(\bG,\bN)$ from (\ref{eq:d:51}).  In \cite[\S 2(iii)]{BFN1}, these authors prove that $\cA(\bG,\bN)$ carries an integer grading whose Hilbert series is given by the monopole formula:  
\begin{equation}
\sum_{\gamma} t^{2 \Delta(\gamma)} P_{\bG,\gamma}(t)
\end{equation}
Here the sum runs over dominant coweights $\gamma$ of $\bG$, and $P_{\bG,\gamma}(t) = \prod_i (1-t^{2d_i})^{-1}$ where the $d_i$ are the exponents of the subgroup $\Stab_{t^\gamma} \subseteq \bG$ defined as in \S \ref{subsection: Minuscule monopole operator}. Importantly for us, the degrees of dressed minuscule monopole operators from \S\ref{subsection: Minuscule monopole operator} are given by
\begin{equation}
\label{appendix 3}
\deg M_\gamma^{\bG,\bN}(f) \ = \ 2 \Delta(\gamma) + \deg f,
\end{equation}
Here $\deg f$ denotes the cohomological degree of $f \in H_\bullet^{\Stab_{t^{\gamma}}}(\pt)$, and in particular $\deg f \geq 0$.

Following the terminology of \cite{GW}, we say that the pair $(\bG,\bN)$ is:
\begin{itemize}
\item \emph{good} if $2\Delta(\gamma) > 1$ for all coweights $\gamma$ of $\bG$;
\item \emph{ugly} if $2 \Delta(\gamma) \geq 1$ for all coweights $\gamma$, and this bound is attained by some $\gamma$.
\end{itemize}
As follows immediately from the monopole formula: the pair $(\bG,\bN)$ is good or ugly if and only if the grading on $\cA(\bG,\bN)$ is \emph{conical}.  Recall that this means that the grading on $\cA(\bG,\bN)$ is only in non-negative degrees, and the degree zero piece consists solely of the ground field $\kk$.  In this case, there is a distinguished cone point $o\in \cM_C(\bG,\bN)$, given by the vanishing of the irrelevant ideal of $\cA(\bG,\bN)$ (generated by elements of degree $\geq 1$).

Finally, recall from \S\ref{subsection: Poisson structure} that $\cA(\bG,\bN)$ carries a Poisson structure.  This structure naturally has degree $-2$, as it is induced by a deformation quantization over the cohomology ring $H_{\CC^\times}^\bullet(pt) =\kk[\hbar] $, where $\hbar$ has degree 2.  Now, when $(\bG,\bN)$ is good, it follows from the monopole formula that not only is $\cA(\bG,\bN)$ conical, but it has no elements of degree one.  The irrelevant ideal is therefore actually generated by elements of degree $\geq 2$, and so is easily seen to be Poisson.  In particular, in this case the corresponding map $\cA(\bG,\bN) \twoheadrightarrow \kk$ is Poisson.

Summarizing the above discussion:
\begin{Lemma}
\label{lemma: first lemma appendix}
If $(\bG,\bN)$ is good or ugly, then $\cM_C(\bG,\bN)$ is conical.  Moreover, if $(\bG,\bN)$ is good, then the cone point $o \in \cM_C(\bG,\bN)$ is a symplectic leaf.
\end{Lemma}

\subsection{The case of quivers}
\label{sec:appendix2}
Consider now the case of a quiver gauge theory defined as in \S\ref{sec:qgt}.  Given $\lambda \in \coweightlattice_{++}$ and $\mu \in \coweightlattice$  with $\lambda \geq \mu$, recall that we associate $\bw,\bv$ by equations (\ref{eq:d:104}) and (\ref{eq:d:104b}), and thus associate $(\bG,\bN)$ as in \eqref{eq:bG-and-bN-for-quiver-gauge-theory} and \eqref{eq:bG-and-bN-for-quiver-gauge-theoryb}. 

Recall that the algebra $\cA(\lambda,\mu) = \cA(\bG,\bN) $ is generated by the FMOs $\M_{\bm}^{\pm}(f) $ where $\mathbf{0}\leq \bm\leq \bv$, see \S \ref{sec: FMOs}. Since these FMOs are precisely the dressed minuscule monopole operators for the minuscule coweights $\pm \varpi_\bm$ of $\bG$, their degrees are given by the formula (\ref{appendix 3}).  Since these elements generate the algebra, to verify whether $(\bG,\bN)$ is good or ugly, it suffices to the check that $2 \Delta(\pm \varpi_\bm) \geq 2$ for all non-zero $\mathbf{0} \leq \bm \leq \bv$ (resp.~$2\Delta(\pm \varpi_\bm) \geq 1$ and this bound is saturated).

\begin{Lemma}
For any $\mathbf{0} \leq \bm \leq \bv$, we have:
$$
2\Delta( \pm \varpi_\bm) =  \ \bm \cdot (\bw - C \bv) + \bm \cdot (C \bm)
$$
\end{Lemma}
This lemma is based on calculations from \cite[\S 5]{NakQuestions}.  We include the brief proof.
\begin{proof}
By definition $\Delta(\gamma) = \Delta(- \gamma)$ for any coweight $\gamma$, so we focus on  $\gamma = \varpi_\bm$.

Let us consider the formula (\ref{appendix 1}) for $\Delta(\varpi_\bm)$. Recall that $\varpi_\bm$ is the sum over $i$ of the $\sm_i$-th fundamental coweights $(1,\ldots,1,0,\ldots,0)$.   Working one $\GL(V_i)$ at a time, it is easy to see that the sum over positive roots $\alpha$ of $\prod_i \GL(V_i)$ is given by
$$
\sum_\alpha | \langle \alpha, \varpi_\bm \rangle | \ = \ \sum_i \sm_i (\bv_i - \sm_i)
$$
Next, the weights of $\bN$ are of two types: for each $i\in I$ there are the weights of  $ \Hom(W_i, V_i)$, and for each  edge $a \in E$ there are the weights of $\Hom(V_{\source(a)}, V_{\target(a)})$.  It is not hard to see that:
$$
\sum_\xi |\langle \xi, \varpi_\bm\rangle | \dim \bN(\xi) \ = \ \sum_{i\in I} \sm_i \sw_i + \sum_{a \in E}  \big( \sm_{\target(a)} (\sv_{\source(a)} - \sm_{\source(a)}) + \sm_{\source(a)} (\sv_{\target(a)} - \sm_{\target(a)})   \big)
$$
Substituting into the formula \eqref{appendix 1},  we therefore have
\begin{align*}
2 \Delta(\varpi_\bm) & = - 2 \sum_i \sm_i (\sv_i - \sm_i) + \sum_{i\in I} \sm_i \sw_i + \sum_{a \in E}  \big( \sm_{\target(a)} (\sv_{\source(a)} - \sm_{\source(a)}) + \sm_{\source(a)} (\bv_{\target(a)}- \sm_{\target(a)})   \big)
\end{align*}
which in turn can be easily rearranged to give:
$$
\sum_i \sm_i \sw_i  - 2 \sum_i \sm_i \sv_i + \sum_{a \in E} (\sm_{\source(a)} \sv_{\target(a)} + \sm_{\target(a)} \sv_{\source(a)}) + 2 \sum_i \sm_i \sm_i - \sum_{a \in E}  (\sm_{\source(a)}\sm_{\target(a)} + \sm_{\target(a)} \sm_{\source(a)})
$$
Observe that this sum does not depend on the orientation of our quiver, and is exactly $\bm \cdot (\bw - C \bv) + \bm \cdot (C \bm)$.
\end{proof}
Summarizing the above discussion, we can finally explain our previous definitions of conicity (Definition \ref{def:conical-coweights}) and goodness (Definition \ref{def:good-coweights}):
\begin{Corollary}
\label{appendix:main result}
If $\mu$ satisfies the conicity condition for $\lambda$ (Definition \ref{def:conical-coweights}), then $(\bG,\bN)$ is good or ugly, and in particular $\cA(\lambda,\mu)$ is conical.  Moreover, if $\mu$ satisfies the good condition for $\lambda$ (Definition \ref{def:good-coweights}), then $(\bG,\bN)$ is good and the map $\cA(\lambda,\mu) \twoheadrightarrow \kk$ is Poisson.
\end{Corollary}

\begin{Remark}
Using the above lemma, it is easy to see that the surjection $\cA(\lambda,\mu) \twoheadrightarrow \cA(\lambda',\mu)$ from Theorem \ref{thm:closed-embedding-of-KM-slices} respects the gradings discussed above.  Indeed, recall that $\M_\bm^\pm (f) \mapsto \M_\bm^\pm(\tilde{f})$ or zero.  Assume the former case. For any $\mathbf{0} \leq \bm$, it follows by construction that
$$
\bm \cdot (\bw - C \bv)  = \bm \cdot (\bw' - C \bv'),
$$
and therefore $\deg \M_\bm^\pm (f) = \deg \M_\bm^\pm(\tilde{f})$ by \eqref{appendix 3} and the lemma.
\end{Remark}

Finally, we conclude by studying the finite and affine type cases.  Recall that in affine type, the \emph{level} of a coweight $\mu$ is its pairing with the imaginary root $\delta$. Following the notation of \cite[Theorem 4.8]{Kac}, recall that we can write $\delta = \sum_{i \in I} a_i \simpleroot_i $ for some positive integers $a_i$.
\begin{Theorem}
\label{thm: conicity finite affine}
Let $\lambda, \mu \in \coweightlattice_{++}$ be dominant coweights with $\lambda > \mu$.  
\begin{enumerate}
\item If the quiver is finite type, then $\mu$ is good for $\lambda$. 

\item If the quiver is affine type, then:
\begin{enumerate}[(a)]
\item If the level of $\mu$ is $\geq 2$, then $\mu$ is good for $\lambda$.

\item If the level of $\mu$ is 1, then $\mu$ satisfies the conicity condition for $\lambda$, but is not good for $\lambda$.

\item If the level of $\mu$ is 0, then $\mu$ does not satisfy the conicity condition for $\lambda$. 
\end{enumerate}
\end{enumerate}
\end{Theorem}
In particular, part (2c) above demonstrates that for general symmetric Kac-Moody types, $\mu$ being dominant does not necessarily imply the conicity condition. 
\begin{proof}

The fact that $\mu$ is dominant means precisely that $\bm \cdot (\bw - C \bv) \geq 0$ for any $\bm \geq \mathbf{0}$.  Note also that we can identify $\bm \cdot (C \bm)$ with the standard bilinear form on the root lattice.  In finite type we therefore have $\bm \cdot (C \bm) \geq 2$ for any non-zero $\bm$.  This implies that $\mu$ is good for $\lambda$, since for any   $\mathbf{0} < \bm \leq \bv$ we  see that
$$
\bm \cdot (\bw - C \bv) + \bm \cdot (C \bm) \geq 0 +  2 = 2
$$
The same argument applies in affine type, except in those cases where $\bm$ corresponds to a multiple of the imaginary root $\delta$, in which case $\bm \cdot (C\bm) = 0$ \cite[Theorem 4.3]{Kac}.  Note that we still require $\mathbf{0} < \bm \leq \bv$, and these $\bm$ corresponding to multiples of $\delta$ are thus only allowed if $a_i \leq \sv_i$ for all $i\in I$.  To complete the proof, we will study the latter case more carefully.

First observe that it suffices to study the case where $\bm$ corresponds exactly to $\delta$, meaning that $\sm_i = a_i$ for all $i\in I$.  In this case $\bm \cdot (\bw - C \bv)$ is exactly the level of $\mu$. So on the one hand, if the level of $\mu$ is $\geq 2$, then we have
$$
\bm \cdot (\bw - C \bv) + \bm \cdot (C \bm) \geq 2 + 0 = 2,
$$
This shows that $\mu$ is good for $\lambda$.  On the other hand, suppose that the level of $\mu$ is 0 or 1.  Then $\lambda$ also has level 0 (resp.~1).  Since $\lambda > \mu$ are both dominant and we are in symmetric affine type, the only possibility is that $\mu = \lambda - k \delta$ for some integer $k> 0$ (this follows from \cite[\S 12.4--12.6]{Kac}).  This automatically implies that the vector $\bm = (a_i)_{i\in I}$ corresponding to $\delta$ satisfies $\mathbf{0} < \bm \leq \bv$, and we have
$$
\bm \cdot (\bw - C \bv) + \bm \cdot(C \bm)  = 0 \ \ \text{(resp.~1)}
$$
Since this is equal to $2 \Delta(\varpi_\bm)$, this shows that $\mu$ does not satisfy the conicity condition for $\lambda$ (resp.~satisfies the conicity condition, but not the good condition).
\end{proof}
